\documentclass[reqno,a4paper, 11pt]{amsart}

\usepackage[english]{babel}
\usepackage[utf8]{inputenc}
\usepackage{amsfonts,epsfig}
\usepackage{latexsym}
\usepackage{amsmath}
\usepackage{amssymb}
\usepackage{mathabx}
\usepackage{comment}
\usepackage{color}
\usepackage{enumerate}
\usepackage{csquotes}
\usepackage[backend=biber,
            style=alphabetic,
            giveninits=true,
            isbn=false,
            doi=false,
            url=false,
            sorting=nyvt]{biblatex}
\renewbibmacro{in:}{}

\newtheorem{theorem}{Theorem}
\newtheorem{lemma}[theorem]{Lemma}
\newtheorem{corollary}[theorem]{Corollary}

\theoremstyle{definition}

\theoremstyle{remark}

\newtheorem{claim}[theorem]{Claim}

\numberwithin{equation}{section}

\newcommand{\set}[1]{\left\{{#1}\right\}}

\newcommand{\measset}[1]{m \left(\set{#1}\right)}
\newcommand{\B}{\mathcal{B}}
\newcommand{\D}{\mathbb{D}}

\newcommand{\N}{\mathbb{N}}
\newcommand{\R}{\mathbb{R}}
\newcommand{\C}{\mathbb{C}}
\newcommand{\A}{\mathcal{A}}
\DeclareMathOperator{\Real}{Re}
\def\circle{\partial \D}
\def\ic{\int_{\partial \D}}

\title{A Central Limit Theorem for Inner Functions}

\author{Artur Nicolau}
\address{Universitat Aut\`onoma de Barcelona, Departament de Matem\`atiques,  08193 Barcelona}
\email{artur@mat.uab.cat}

\author{Od\'i Soler i Gibert}
\address{Universitat Aut\`onoma de Barcelona, Departament de Matem\`atiques,  08193 Barcelona}
\email{odisoler@mat.uab.cat}

\thanks{Both authors are supported in part by the Generalitat de Catalunya (grant 2017 SGR 395)
        and the Spanish Ministerio de Ci\'encia e Innovaci\'on (project  MTM2017-85666-P)}
        
\begin{document}

    \begin{abstract}
        A Central Limit Theorem for linear combinations of iterates of an inner function is proved.
        The main technical tool is Aleksandrov Desintegration Theorem for Aleksandrov-Clark measures.
    \end{abstract}

    \maketitle

    \section{Introduction and main results}
    \label{sec:Introduction}

    Inner functions are analytic mappings from the unit disc $\D$ into itself
    whose radial limits are of modulus one at almost every point of the unit circle $\partial \D$.
    Inner functions were introduced by R. Nevanlinna and after the pioneering work of brothers Riesz,
    Frostmann and Beurling, they have become a central notion in Analysis.
    See for instance \cite{ref:GarnettBoundedAnalyticFunctions}.
    Any inner function $f$ induces a mapping from the unit circle into itself defined
    at almost every point $z \in \circle$ by $f(z) = \lim_{r \to 1} f(rz)$.
    This boundary mapping will be also called $f$.
    It is well known that if $f(0)=0$,
    normalized Lebesgue measure $m$ in the unit circle is invariant under this mapping,
    that is, $m(f^{-1} (E)) = m(E)$ for any measurable set $E \subset \circle$.
    Several authors have also studied the distortion of Hausdorff measures by this mapping.
    See \cite{ref:FernandezPestanaDistortionInnerFunctions} and \cite{ref:LeviNicolauSoler}.
    Dynamical properties of the mapping $f\colon \circle \rightarrow \circle$,
    as recurrence, ergodicity, mixing, entropy and others have been studied
    by Aaranson \cite{ref:AaronsonErgodicInnerFunctions}, Crazier \cite{ref:CraizerEntropyInnerFunctions},
    Doering and Ma\~{n}e \cite{ref:DoeringMane},
    Fern\'andez, Meli\'an and Pestana \cite{ref:FernandezMelianPestanaInnerFunctionsMixing},
    \cite{ref:FernandezMelianPestanaExpandingMaps},
    Neurwirth \cite{ref:NeuwirthErgodicity}, Pommerenke \cite{ref:PommerenkeErgodicInnerFunctions}, and others.
    Dynamical properties of inner functions have been recently used in several problems
    on the dynamics of meromorphic functions in simply connected Fatou components.
    See \cite{ref:BaranskiFagellaJarqueKarpinskaAccessesInfinity},
    \cite{ref:BaranskiFagellaJarqueKarpinskaEscapingPoints}
    and \cite{ref:EvdoridouFagellaJarqueSixsmithInnerFunctionSingularities}.

    It is well known that in many senses lacunary series behave as sums of independent random variables.
    Salem and Zygmund (\cite{ref:SalemZygmundLacunarySeries-I} and \cite{ref:SalemZygmundLacunarySeries-II})
    proved a version of the Central Limit Theorem for lacunary series and, a few years later,
    Weiss proved a version of the Law of the Iterated Logarithm in this context (\cite{ref:WeissLILLacunarySeries}).
    Our main result is a Central Limit Theorem for linear combinations of iterates of an inner function fixing the origin.
    It is worth mentioning that in our result no lacunarity assumption is needed.
    Recall that a sequence of measurable functions $\{f_N \}$ defined
    at almost every point in the unit circle converges in distribution
    to a (circullary symmetric) standard complex normal variable if for any measurable set $K \subset \C$, one has 
    \begin{equation*}
        \lim_{N \to \infty} \measset{z \in \circle\colon f_N (z) \in K}
        = \frac{1}{ 2 \pi} \int_K e^{- |w|^2 / 2}\, dA(w) . 
    \end{equation*} 
    As it is usual we denote by $f^n$ the $n$-th iterate of the function $f$.  

    \begin{theorem}
        \label{thm:MainTheorem}
        Let $f$ be an inner function with $f(0)=0$ which is not a rotation.
        Let $\{a_n\}$ be a sequence of complex numbers.
        Consider 
        \begin{equation}
            \label{eq:NSumVariance}
            \sigma_N^2
            = \sum_{n=1}^{N} |a_n|^2 + 2 \Real \sum_{k=1}^N f'(0)^k \sum_{n=1}^{N-k} \overline{a_n} a_{n+k},
            \quad  N=1,2,\ldots
        \end{equation}
        Assume there exists a constant $\eta >0$ such that 
        \begin{equation}
            \label{eq:CoefficientGrowthHpth}
            \lim_{N \to \infty}  \frac{ \sup \set{|a_n|^2\colon n \leq N} }{\left(\sum_{n=1}^N |a_n|^2\right)^{(1- \eta)/2}} =0. 
        \end{equation}
        Then 
        \begin{equation*}
            \frac{1}{ \sqrt{2} \sigma_N} \sum_{n=1}^{N} a_n f^n
        \end{equation*}
        converges in distribution to a standard complex normal variable.  
    \end{theorem}

    Let $f$ be an inner function fixing the origin.
    Then it is well known that Lebesgue measure $m$ is ergodic.
    Hence the classical Ergodic Theorem gives that
    \begin{equation*}
        \lim_{N \to \infty}  \frac{1}{N} \sum_{n=1}^{N}  f^n (z) = 0  
    \end{equation*}
    at almost every point $z \in \circle$.
    This can be understood as a version of the Law of Large Numbers.
    Our result provides the corresponding version of the Central Limit Theorem.
    Actually taking $a_n = 1$, $n=1,2\ldots$, in Theorem \ref{thm:MainTheorem},
    one can easily show that $\lim_{N \to \infty} \sigma_N^2 / N \sigma^2 = 1$, where
    \begin{equation}
        \label{eq:CoefficientOneVariance}
        \sigma^2 = \Real \frac{1+ f'(0)}{1 - f'(0)}
    \end{equation}
    and we deduce the following result. 


    \begin{corollary}
        \label{Coro1}
        Let $f$ be an inner function with $f(0)=0$ which is not a rotation.
        Then 
        \begin{equation*}
            \frac{1}{\sqrt{2N}} \sum_{n=1}^N f^n 
        \end{equation*}
        converges in distribution to a complex normal variable with mean $0$
        and variance $\sigma^2$ given by \eqref{eq:CoefficientOneVariance},
        that is, for any measurable set $K \subset \C$, we have
        \begin{equation*}
            \lim_{N \to \infty} \measset{z \in \circle\colon (2N)^{-1/2} \sum_{n=1}^{N} f^n (z) \in K}
            = \frac{1}{ 2 \pi \sigma^2} \int_K e^{- |w|^2 / 2 \sigma}\, dA(w) . 
        \end{equation*}
    \end{corollary}

    Observe that when $f'(0)$ is close to $1$ and hence $f$ is close to be the identity map,
    the variance $\sigma^2$ is large.
    However if $f'(0)$ is close to a unimodular constant different from $1$, the variance is small.
    On the opposite side, if $f'(0)=0$, $\sigma= 1$. 

    Let $H^2$ be the Hardy space of analytic functions in $\D$ whose Taylor coefficients are square summable.
    Let $\{a_n\}$ be a sequence of complex numbers.
    It is easy to show (see Theorem \ref{thm:L2L4NormEstimates}) that $\sum_n a_n f^n$ converges in $H^2$ if and only if $\sum_n |a_n|^2 < \infty$.
    A repetition of the proof of our main result gives the following version of the Central Limit Theorem for the tails. 

    \begin{theorem}
        \label{thm:TailsTheorem}
        Let $f$ be an inner function with $f(0)=0$ which is not a rotation.
        Let $\{a_n\}$ be a square summable sequence of complex numbers.
        Consider 
        \begin{equation}
            \label{eq:TailsVariance}
            \sigma^2 (N)
            = \sum_{n \geq N} |a_n|^2 + 2 \Real \sum_{k \geq 1} f'(0)^k \sum_{n \geq N} \overline{a_n} a_{n+k},
            \quad  N=1,2,\ldots
        \end{equation}
        Assume there exists a constant $\eta >0$ such that
        \begin{equation}
            \label{eq:TailsGrowthHpth}
            \lim_{N \to \infty}  \frac{ \sup \set{ |a_n|^2\colon n \geq N} }{\left(\sum_{n \geq N}  |a_n|^2\right)^{(1- \eta)/2}} =0. 
        \end{equation}
        Then 
        \begin{equation*}
            \frac{1}{ \sqrt{2} \sigma (N)} \sum_{n=N}^{\infty} a_n f^n
        \end{equation*}
        converges in distribution to a standard complex normal variable. 
    \end{theorem}

    Let $S_N^2 = \sum_{n=1}^N |a_n|^2$.
    It is easy to show (see Theorem \ref{thm:L2L4NormEstimates}) that there exists a constant $C=C(f)>0$
    such that $C^{-1} S_N^2 \leq \sigma_N^2 \leq C S_N^2$, $N=1,2, \ldots$.
    When $f'(0)=0$ we have ${\sigma}_N = S_N$ but in general, both quantities do not coincide.
    However if the following uniform quasiorthogonality condition holds
    \begin{equation}
        \label{eq:Quasiorthogonality}
        \lim_{N \to \infty} \frac{\sup_{k \leq N} \left|\sum_{n=1}^{N-k} \overline{a_n} a_{n+k}\right|  }{S_N^2} =0, 
    \end{equation}
    then $\lim_{N \to \infty} {S_N}/\sigma_N = 1$ and Theorem \ref{thm:MainTheorem} gives that
    \begin{equation*}
        \frac{1}{\sqrt{2} S_N} \sum_{n=1}^N a_n f^n
    \end{equation*}
    converges in distribution to a standard complex normal variable. 

    We now make some remarks on the assumption and proof of Theorem \ref{thm:MainTheorem}.
    Condition \eqref{eq:CoefficientGrowthHpth} implies that $\sum_n |a_n|^2 = \infty$,
    but one can not expect this last condition to be sufficient in Theorem \ref{thm:MainTheorem}.
    However note that if $\{a_n\}$ is bounded, both conditions are equivalent.
    The proof of Theorem \ref{thm:MainTheorem} uses two relevant properties of the iterates of an inner function fixing the origin.
    The first one is that the square of the modulus of the partial sums are uncorrelated.
    More concretely, given a set $\A$ of positive integers, consider the corresponding partial sum 
    \begin{equation*}
        \xi (\A) = \sum_{n \in \A} a_n f^n.
    \end{equation*}
    If $\A \cap \B = \emptyset$, we will show in Theorem \ref{thm:SquaredModuliPartialSums} that 
    \begin{equation}
        \label{eq:PartialSumsUncorrelated}
        \ic |\xi (\A)|^2|\xi (\B)|^2\, dm = \left(\ic |\xi (\A)|^2\, dm\right) \left(\ic |\xi (\B)|^2\, dm\right).
    \end{equation}
    The second property provides an exponential decay of the higher order correlations of the iterates.
    More concretely, let $\varepsilon_i = 1$ or $\varepsilon_i = -1$ for $i=1,2,\ldots, k$
    and $n_1 < \ldots < n_k$ be positive integers satisfying $n_j - n_{j-1} \geq q \geq 1$, $j=2, \ldots , k$.
    Denote $\boldsymbol{\varepsilon} = (\varepsilon_1,\ldots,\varepsilon_k)$
    and $\boldsymbol{n} = (n_1,\ldots,n_k).$
    For a positive integer $n$, denote by $f^{-n}$ the function defined by $f^{-n} (z) = \overline{f^n (z)}$, $z \in \circle$.
    We will prove in Theorem \ref{thm:HigherCorrelations} that there exists a constant $C>0$, independent of the indices, such that  
    \begin{equation}
        \label{eq:HigherOrderCorrelations}
        \left|\ic \prod_{j=1}^k f^{\varepsilon_j n_j}\, dm\right|
        \leq C^k k! |f'(0)|^{\Phi(\boldsymbol{\varepsilon},\boldsymbol{n})}, \quad k=1,2, \ldots ,
    \end{equation}
    if $q$ is sufficiently large
    and where $\Phi$ is a certain function depending on the choice of indices
    that satisfies $\Phi(\boldsymbol{\varepsilon},\boldsymbol{n}) \geq kq/4$.
    The main technical tool in the proof of both properties
    \eqref{eq:PartialSumsUncorrelated} and \eqref{eq:HigherOrderCorrelations} is
    the theory of  Aleksandrov-Clark measures and more concretely, the Aleksandrov Desintegration Theorem. 

    The paper is organized as follows.
    In Section \ref{sec:ACMeasures} we introduce Aleksandrov-Clark measures and use them to prove property \eqref{eq:PartialSumsUncorrelated}.
    In Section \ref{sec:NormsPartialSums} we estimate the $L^2$ and the $L^4$ norm of $\xi(A)$.
    In Section \ref{sec:HigherOrderCorrel} we prove estimate \eqref{eq:HigherOrderCorrelations}.
    The proof of Theorem \ref{thm:MainTheorem} is given in Section \ref{sec:ProofMainThm}.   

    \section{Alekandrov-Clark measures and Property \texorpdfstring{\eqref{eq:PartialSumsUncorrelated}}{(\ref{eq:PartialSumsUncorrelated})}}
    \label{sec:ACMeasures}
    We start with an elementary auxiliary result which is just a restatement of the invariance of Lebesgue measure. 

    \begin{lemma}
        \label{lemma:DeconjugateInnerFunction}
        Let $f$ be an inner function with $f(0)=0$. 
        \begin{enumerate}[(a)]
            \item
            \label{lemma:item:DeconjugateInnerFunction}
            Let $G$ be an integrable function on $\circle$.
            Then
            \begin{equation*}
                \ic G(f(z))\, dm(z) = \ic G(z)\, dm(z)
            \end{equation*}
            
            \item
            \label{lemma:item:IteratesCovariance}
            Let $k<j$ be positive integers.
            Then 
            \begin{equation*}
                \ic \overline{f^k} f^j\, dm = f'(0)^{j-k}
            \end{equation*}
        \end{enumerate}
    \end{lemma}
    \begin{proof}[Proof of Lemma~\ref{lemma:DeconjugateInnerFunction}]
        We can assume that $G$ is the characteristic function of a measurable set $E \subset \circle$.
        Since $m(f^{-1} (E)) = m (E)$, the identity \eqref{lemma:item:DeconjugateInnerFunction} follows.
        Using \eqref{lemma:item:DeconjugateInnerFunction} and Cauchy formula, we have
        \begin{equation*}
            \ic \overline{f^k} f^j\, dm = \ic \overline{z} f^{j-k} (z)\, dm(z) = f'(0)^{j-k}.
        \end{equation*}
    \end{proof}

    Given an analytic mapping from the unit disc into itself and a point $\alpha \in \circle$,
    the function $(\alpha + f)/(\alpha - f)$ has positive real part
    and hence there exists a positive measure $\mu_\alpha = \mu_\alpha (f)$ in the unit circle
    and a constant $C_{\alpha} \in \R$ such that
    \begin{equation}
        \label{eq:ACMeasureHerglotz}
        \frac{\alpha + f(w)}{\alpha - f(w)} = \ic \frac{z + w}{z - w}\, d\mu_\alpha(z) + iC_{\alpha}, \quad w \in \D. 
    \end{equation}
    The measures $\{\mu_\alpha\colon \alpha \in \circle\}$ are called the Aleksandrov-Clark measures of the function $f$.
    Clark introduced them in his paper \cite{ref:ClarkOneDimensionalPerturbations}
    and many of their deepest properties were found by Aleksandrov in \cite{ref:AleksandrovMeasurablePartitionsCircle},
    \cite{ref:AleksandrovMultiplicityBoundaryValuesInnerFunctions} and \cite{ref:AleksandrovInnerFunctionsAndRelatedSpaces}.
    The two surveys \cite{ref:PoltoratskiSarasonACMeasures} and \cite{ref:SaksmanACMeasures}
    as well as \cite[Chapter~IX]{ref:CimaMathesonRossCauchyTransform} contain their main properties and a wide range of applications.
    Observe that if $f(0)=0$ then $\mu_\alpha$ are probability measures.
    Moreover, $f$ is inner if and only if $\mu_\alpha$ is a singular measure for some (all) $\alpha \in \circle$.
    Assume $f(0)=0$.
    Computing the first two derivatives in formula \eqref{eq:ACMeasureHerglotz} and evaluating at the origin, we obtain 
    \begin{equation}
        \label{eq:FirstACMeasureMoment}
        \ic z\, d\mu_{\alpha}(z) = \overline{f'(0)} \alpha, \quad \alpha \in \circle, 
    \end{equation}
    and
    \begin{equation}
        \label{eq:SecondACMeasureMoment}
        \ic z^2\, d\mu_{\alpha}(z) = \overline{\frac{f''(0)}{2}} \alpha + \overline{f'(0)}^2 {\alpha}^2, \quad \alpha \in \circle.
    \end{equation}
    Our main technical tool is Aleksandrov Desintegration Theorem which asserts that 
    \begin{equation}
        \label{eq:AleksandrovDesintegrationThm}
        m = \ic \mu_\alpha\, dm(\alpha)
    \end{equation}
    holds true in the sense that
    \begin{equation*}
        \ic G\, dm = \ic \ic G(z)\, d\mu_\alpha(z)\, dm(\alpha), 
    \end{equation*}
    for any integrable function $G$ on the unit circle.
    Aleksandrov Desintegration Theorem will be used in our next auxiliary result. 

    \begin{lemma}
        \label{lemma:IteratesProduct}
        Let $f$ be an inner function with $f(0)=0$.
        For $k =1,2,\ldots,p$, let $n_k, j_k$,  be positive integers such that 
        \begin{equation}
            \label{eq:SeparationHpthTwoIteratesProd}
            \max \{ n_k , j_k \} < \min \{n_{k+1} , j_{k+1} \} , \quad  k=1,\ldots,p-1 .
        \end{equation}
        Then 
        \begin{equation}
            \label{eq:IteratesProduct}
            \int_{\circle} \prod_{k=1}^p f^{n_k} \overline{f^{j_k}} dm = \prod_{k=1}^p \int_{\circle}  f^{n_k} \overline{f^{j_k}} dm . 
        \end{equation}
    \end{lemma}
    \begin{proof}[Proof of Lemma~\ref{lemma:IteratesProduct}]
        We argue by induction on $p$.
        Assume \eqref{eq:IteratesProduct} holds for $p-1$ products.
        We can assume $n_1 < j_1$.
        By part \eqref{lemma:item:DeconjugateInnerFunction} of Lemma \ref{lemma:DeconjugateInnerFunction} we have
        \begin{equation*}
            \ic \prod_{k=1}^p f^{n_k} \overline{f^{j_k}}\, dm
            = \ic z \overline{f^{j_1 - n_1} (z)} \prod_{k=2}^p f^{n_k - n_1} (z) \overline{f^{j_k - n_1} (z)}\, dm(z).  
        \end{equation*}
        Let $\{\mu_\alpha\colon \alpha \in \circle \}$ be the Aleksandrov-Clark measures of the inner function $f^{j_1 - n_1}$.
        The Aleksandrov Desintegration Theorem \eqref{eq:AleksandrovDesintegrationThm} gives that last integral can be written as
        \begin{equation*}
            \ic \ic z \overline{\alpha} \prod_{k=2}^p f^{n_k - j_1}(\alpha) \overline{f^{j_k - j_1}(\alpha)}\, d\mu_\alpha(z)\, dm(\alpha). 
        \end{equation*}
        By \eqref{eq:FirstACMeasureMoment} and part \eqref{lemma:item:IteratesCovariance} of Lemma \ref{lemma:DeconjugateInnerFunction},
        we have
        \begin{equation*}
            \ic z\, d\mu_\alpha(z) = {\overline{f'(0)}}^{j_1 - n_1} \alpha
            = \alpha \ic f^{n_1} \overline{f^{j_1}}\, dm . 
        \end{equation*}
        Hence
        \begin{equation*}
            \ic \prod_{k=1}^p f^{n_k} \overline{f^{j_k}}\, dm
            = \left(\ic f^{n_1} \overline{f^{j_1}}\, dm\right) \ic \prod_{k=2}^p f^{n_k - j_1} \overline{f^{j_k - j_1}}\, dm
        \end{equation*}
        and we can apply the inductive assumption.
        The invariance property of part \eqref{lemma:item:DeconjugateInnerFunction}
        of Lemma \ref{lemma:DeconjugateInnerFunction} finishes the proof.
    \end{proof}

    Our next result is the first important tool in the proof of Theorem \ref{thm:MainTheorem}.

    \begin{theorem}
        \label{thm:SquaredModuliPartialSums}
        Let $f$ be an inner function with $f(0)=0$.
        Let ${\A}_k$, $k =1,2,\ldots,p$, be finite collections of positive integers such that 
        \begin{equation}
            \label{eq:SeparationHpthPartialSums}
            \max \{ n\colon n \in {\A}_k \} < \min \{n\colon n \in {\A}_{k+1} \} , \quad k=1,\ldots , p-1.
        \end{equation}
        Consider 
        \begin{equation*}
            \xi_k = \sum_{n \in {\A}_k} a_n f^n.
        \end{equation*}
        Then 
        \begin{equation*}
            \ic \prod_{k=1}^p |\xi_k|^2\, dm = \prod_{k=1}^p \ic  |\xi_k|^2\, dm . 
        \end{equation*}
    \end{theorem}
    \begin{proof}[Proof of Theorem~\ref{thm:SquaredModuliPartialSums}]
        Al almost every point of the unit circle we have
        \begin{equation*}
            |\xi_k|^2
            = \sum_{n \in {\A}_k} |a_n|^2 + \sum( \overline{a_n} a_j \overline{f^n} f^j + \overline{a_j} a_n \overline{f^j} f^n ),
        \end{equation*}
        where the last sum is taken over all indices $n,j \in {\A}_k$ with $j>n$.
        Hence $\prod |\xi_k|^2$ can be written as a linear combination of terms of the form 
        \begin{equation*}
            \prod f^{n_k} \overline{f^{j_k}},
        \end{equation*}
        where $n_k, j_k \in {\A}_k$.
        Observe that \eqref{eq:SeparationHpthPartialSums} gives
        the assumption \eqref{eq:SeparationHpthTwoIteratesProd} in Lemma \ref{lemma:IteratesProduct}.
        Now Lemma \ref{lemma:IteratesProduct} finishes the proof.
    \end{proof}

    \section{Norms of Partial Sums}
    \label{sec:NormsPartialSums}
    In this Section we will use Aleksandrov-Clark measures to estimate the $L^2$ and $L^4$ norms
    of linear combinations of iterates of an inner function fixing the origin.
    The main result of this Section is Theorem \ref{thm:L2L4NormEstimates}.
    It is worth mentioning that the asymptotic behavior of the Aleksandrov-Clark measures
    of iterates of an inner function has been studied in \cite{ref:GallardoNieminenACMeasures},
    but we will not use their results.
    As before, if $n$ is a positive integer, we will use the notation $f^{-n}$ to denote
    the function defined by $f^{-n} (z) = \overline{f^n (z)}$, for almost every $z \in \circle$.
    We start with a technical auxiliary result which will be used later. 

    \begin{lemma}
        \label{lemma:FourFactors}
        Let $f$ be an inner function with $f(0)=0$ which is not a rotation.
        Let $\varepsilon_k = 1$ or $\varepsilon_k = -1$ , $k =1,2,3,4$.
        \begin{enumerate}[(a)]
            \item
            \label{lemma:item:FourFactorsCancellation}
            Let $n_k$, $k =1,2,3,4$, be positive integers with $\max \{n_1 , n_2 \} < \min \{ n_3 , n_4 \}$.
            Then
            \begin{equation*}
                I = I(\varepsilon_1 n_1,-\varepsilon_1 n_2,n_3,n_4)
                = \ic f^{\varepsilon_1 n_1}f^{-\varepsilon_1 n_2}f^{n_3}f^{n_4}\, dm = 0. 
            \end{equation*}
            
            \item
            \label{lemma:item:SecondFactorSquared}
            Let $n_1 < n_2 < n_3$ be positive integers and 
            \begin{equation*}
                II = II(\varepsilon_1 n_1,\varepsilon_2 n_2,\varepsilon_3 n_3)
                = \ic f^{\varepsilon_1 n_1}  (f^{\varepsilon_2 n_2})^2 f^{\varepsilon_3 n_3}\, dm.
            \end{equation*}
            Then there exists a constant $C=C(f) >0$ independent of the indices $n_1, n_2, n_3$,
            such that $|II| \leq C |f'(0)|^{n_3 - n_1}$.
            
            \item
            \label{lemma:item:FirstFactorSquared}
            Let $n_1 < n_2 < n_3$ be positive integers and 
            \begin{equation*}
                III = III(\varepsilon_1 n_1 , \varepsilon_2 n_2 , \varepsilon_3 n_3)
                = \ic (f^{\varepsilon_1 n_1})^2  f^{\varepsilon_2 n_2} f^{\varepsilon_3 n_3}\, dm.
            \end{equation*}
            Then there exists a constant $C=C(f) >0$ independent of the indices $n_1, n_2, n_3$,
            such that $|III| \leq 1$ if $n_2 = n_1 + 1$ and $n_3 \leq n_2 + 2$,
            and  $|III| \leq C |f'(0)|^{n_3 - n_1}$ otherwise.
            
            \item
            \label{lemma:item:FourArbitraryFactors}
            Let $n_1 < n_2 < n_3 < n_4$ be positive integers and 
            \begin{equation*}
                IV = IV(\varepsilon_1 n_1 , \varepsilon_2 n_2 , \varepsilon_3 n_3 , \varepsilon_4 n_4)
                = \ic f^{\varepsilon_1 n_1} f^{\varepsilon_2 n_2} f^{\varepsilon_3 n_3} f^{\varepsilon_4 n_4}\, dm. 
            \end{equation*}
            Then there exists a constant $C=C(f) >0$ independent of the indices $n_1, n_2, n_3, n_4$,
            such that $|IV| \leq C |f'(0)|^{n_2 - n_1 + n_4 - n_3}$ if $n_4 - n_3 >2$,
            and $|IV| \leq C |f'(0)|^{n_3 - n_1}$ if $n_4 - n_3 \leq 2$.
            Moreover $|IV| = |f'(0)|^{n_2 - n_1 + n_4 - n_3}$ if $\varepsilon_1 \varepsilon_2 = \varepsilon_3\varepsilon_4 = -1$.
        \end{enumerate}
    \end{lemma}
    \begin{proof}[Proof of Lemma~\ref{lemma:FourFactors}]
        Let $C$ denote a positive constant which may depend on the function $f$ but not on the indices $\{n_i\}$,
        whose value may change from line to line. 

        \eqref{lemma:item:FourFactorsCancellation}
        We can assume that $n_1 < n_2$.
        Part \eqref{lemma:item:DeconjugateInnerFunction} of Lemma \ref{lemma:DeconjugateInnerFunction} gives that 
        \begin{equation*}
            I = \ic z^{\varepsilon_1} f^{- \varepsilon_1 (n_2 - n_1)} (z)f^{ n_3 - n_1} (z)f^{ n_4 - n_1} (z)\, dm (z). 
        \end{equation*}
        Let $\{\mu_\alpha\colon \alpha \in \circle\}$ be the Aleksandrov-Clark measures of $f^{n_2 - n_1}$.
        The Aleksandrov Desintegration Theorem \eqref{eq:AleksandrovDesintegrationThm} gives 
        \begin{equation*}
            I = \ic\ic z^{\varepsilon_1} {\alpha}^{- \varepsilon_1} f^{ n_3 - n_2} (\alpha)f^{ n_4 - n_2} (\alpha)
            \, d\mu_\alpha(z)\, dm(\alpha). 
        \end{equation*}
        By \eqref{eq:FirstACMeasureMoment}
        \begin{equation*}
            \ic z^{\varepsilon_1}\, d\mu_\alpha(z) = a {\alpha}^{\varepsilon_1},
            \quad \alpha \in \circle, 
        \end{equation*}
        where $|a| = |f'(0)|^{n_2 - n_1}$.
        Since $f(0)=0$, we deduce  
        \begin{equation*}
            |I| = |f'(0)|^{n_2 - n_1} \left|\ic f^{ n_3 - n_2} (\alpha) f^{ n_4 - n_2} (\alpha)\, dm(\alpha)\right| = 0.
        \end{equation*}

        \eqref{lemma:item:SecondFactorSquared}
        We can assume $\varepsilon_1 = 1$.
        Part \eqref{lemma:item:DeconjugateInnerFunction} of Lemma \ref{lemma:DeconjugateInnerFunction} gives that 
        \begin{equation*}
            II = \ic z (f^{ \varepsilon_2 (n_2 - n_1)} (z))^2 f^{\varepsilon_3 ( n_3 - n_1)} (z)\, dm (z). 
        \end{equation*}
        Let $\{\mu_\alpha\colon \alpha \in \circle\}$ be the Aleksandrov-Clark measures of $f^{n_2 - n_1}$.
        The Aleksandrov Desintegration Theorem \eqref{eq:AleksandrovDesintegrationThm} gives 
        \begin{equation*}
            II = \ic\ic z {\alpha}^{2 \varepsilon_2} f^{ \varepsilon_3 (n_3 - n_2)} (\alpha)\, d\mu_\alpha(z)\, dm(\alpha). 
        \end{equation*}
        By \eqref{eq:FirstACMeasureMoment}
        \begin{equation*}
            \ic z\, d\mu_\alpha(z) = {\overline{f'(0)}}^{n_2 - n_1} \alpha,
            \quad \alpha \in \circle .  
        \end{equation*}
        Hence 
        \begin{equation*}
            II = { \overline{f'(0)}}^{n_2 - n_1} \ic {\alpha}^{1 + 2 \varepsilon_2} f^{ \varepsilon_3 (n_3 - n_2)} (\alpha)\, dm (\alpha)
        \end{equation*}
        Since $1 + 2 \varepsilon_2 \leq 3$, the modulus of last integral is bounded by $C|f'(0)|^{n_3 - n_2}$
        if $n_3 - n_2>2$ and by $1$ otherwise.
        This proves \eqref{lemma:item:SecondFactorSquared}. 

        \eqref{lemma:item:FirstFactorSquared}
        We can assume $\varepsilon_1 = 1$.
        Applying part \eqref{lemma:item:DeconjugateInnerFunction} of Lemma \ref{lemma:DeconjugateInnerFunction}
        and Aleksandrov Desintegration Theorem as before, we have 
        \begin{equation*}
            III = \ic \ic z^2 {\alpha}^{ \varepsilon_2} f^{ \varepsilon_3 (n_3 - n_2)} (\alpha)\, d\mu_\alpha(z)\, dm(\alpha), 
        \end{equation*}
        where $\{\mu_\alpha\colon \alpha \in \circle\}$ are the Aleksandrov-Clark measures of $g= f^{n_2 - n_1}$.
        Applying  \eqref{eq:SecondACMeasureMoment}, we obtain 
        \begin{equation*}
            III =  \overline{\frac{g''(0)}{2}} \ic {\alpha}^{1 +  \varepsilon_2} f^{ \varepsilon_3 (n_3 - n_2)} (\alpha)\, dm (\alpha)
            +  \overline{g'(0)^2} \ic {\alpha}^{2 +  \varepsilon_2} f^{ \varepsilon_3 (n_3 - n_2)} (\alpha)\, dm (\alpha). 
        \end{equation*}
        Since $2 +  \varepsilon_2 \leq 3$, both integrals are bounded by $C|f'(0)|^{n_3 - n_2}$ if $n_3 - n_2>2$,
        and by $1$ if $n_3 - n_2 \leq 2$.
        If $n_2 - n_1 >1$, we have that $|g''(0)|/2 + |g'(0)^2| \leq C |f'(0)|^{n_2 - n_1}$.
        If $n_2 - n_1 =1$, we have that $|g''(0)|/2 + |\overline{g'(0)^2}| \leq 2$. 
        This proves \eqref{lemma:item:FirstFactorSquared}. 

        \eqref{lemma:item:FourArbitraryFactors}
        We can assume $\varepsilon_1 = 1$.
        Arguing as before we have
        \begin{equation*}
            IV =
            {\overline{f'(0)}}^{n_2 - n_1} \ic {\alpha}^{1 +\varepsilon_2} f^{\varepsilon_3 (n_3 - n_2)}(\alpha) f^{\varepsilon_4 (n_4 - n_2)}(\alpha)
            \,dm(\alpha)
        \end{equation*}
        If $ \varepsilon_2 = -1$, we repeat the argument and prove that $|IV| \leq |f'(0)|^{n_2 - n_1 + n_4 - n_3}$.
        Moreover if $\varepsilon_2 = -1$ and if $\varepsilon_3 \varepsilon_4 = -1 $,
        we have $|IV| = |f'(0)|^{n_2 - n_1 + n_4 - n_3}$, as stated in the last part of \eqref{lemma:item:FourArbitraryFactors}.
        If $\varepsilon_2 = 1$, let $\{\mu_\alpha\colon \alpha \in \circle\}$
        be the Aleksandrov-Clark measures of $g= f^{n_3 - n_2}$.
        The Aleksandrov Desintegration Theorem \eqref{eq:AleksandrovDesintegrationThm} gives that last integral can be written as  
        \begin{equation}
            \label{eq:TwiceADT}
            \ic \ic z^{2} {\alpha}^{ \varepsilon_3} f^{ \varepsilon_4 (n_4 - n_3)} (\alpha)\, d\mu_\alpha(z)\, dm(\alpha). 
        \end{equation}
        By \eqref{eq:SecondACMeasureMoment}
        \begin{equation*}
            \ic z^2\, d\mu_\alpha(z) = \overline{\frac{g''(0)}{2}} \alpha + \overline{g'(0)^2} {\alpha}^2,
            \quad \alpha \in \circle .  
        \end{equation*}
        Hence the double integral in \eqref{eq:TwiceADT} can be written as  
        \begin{equation*}
            {\overline{\frac{g''(0)}{2}}} \ic {\alpha}^{1 +  \varepsilon_3} f^{ \varepsilon_4 (n_4 - n_3)} (\alpha)\, dm(\alpha)
            + \overline{g'(0)^2} \ic {\alpha}^{2 +  \varepsilon_3} f^{ \varepsilon_4 (n_4 - n_3)} (\alpha)\, dm (\alpha). 
        \end{equation*}
        Since $2 + \varepsilon_3 \leq 3$, both integrals are bounded by $C|f'(0)|^{n_4 - n_3}$ if $n_4 - n_3>2$,
        and by $1$ if $n_4 - n_3 \leq 2$.
        If $n_3 - n_2 >1$, we have that $|g''(0)|/2 + |g'(0)^2| \leq C |f'(0)|^{n_3 - n_2}$. 
        If $n_3 - n_2 =1$, we just use the trivial estimate $|g''(0)|/2 + |g'(0)^2| \leq 2$.
        This proves \eqref{lemma:item:FourArbitraryFactors}.
    \end{proof}

    We will now prove an elementary auxiliary result which will be used several times. 
    \begin{lemma}
        \label{lemma:AuxiliaryBound}
        Let $\A$ be a collection of positive integers and let $\{a_n \}$ be a sequence of complex numbers.
        Fix $\lambda \in \C$ with $|\lambda| < 1$.
        Then  
        \begin{equation*}
            \Bigg|\sum_{n, k \in \A , k>n} \overline{a_n } a_k {\lambda}^{k-n}\Bigg|
            \leq \frac{|\lambda|}{1- |\lambda|} \sum_{n \in \A} |a_n|^2 .   
        \end{equation*}
    \end{lemma}
    \begin{proof}[Proof of Lemma~\ref{lemma:AuxiliaryBound}]
        Writing $j=k-n$ we have that 
        \begin{equation*}
            \sum_{n, k \in \A , k>n} \overline{a_n } a_k {\lambda}^{k-n}
            = \sum_{j>0} {\lambda}^j \sum_{n, n+j \in \A} \overline{a_n } a_{n+j}, 
        \end{equation*}
        where the last sum is taken over all indices $n \in \A$ such that $n+j \in \A$.
        It is also understood that this sum vanishes if there is no $n \in \A$ such that $n+j \in \A$.
        By Cauchy-Schwarz's inequality, 
        \begin{equation*}
            \left|\sum_{n, n+j \in \A} \overline{a_n } a_{n+j}\right|
            \leq \sum_{n \in \A} |a_n|^2. 
        \end{equation*}
        This finishes the proof.
    \end{proof}

    Let $H^2$ be the Hardy space of analytic functions in the unit disc $g(w) = \sum_{n \geq 0} a_n w^n$,  $w \in \D$,
    such that
    \begin{equation*}
        \|g\|_2^2 = \sup_{r<1} \ic |g(rz)|^2\, dm(z) = \sum_{n=0}^\infty  |a_n|^2 < \infty.
    \end{equation*}
    Any function $g \in H^2$ has a finite radial limit $g(z)= \lim_{r \to 1} g(rz)$ at almost every $z \in \circle$ and 
    \begin{equation*}
        \|g\|_2^2 = \ic |g(z)|^2\, dm(z). 
    \end{equation*}
    See \cite{ref:GarnettBoundedAnalyticFunctions}.
    For $0<p<\infty$ let $\|g\|_p$ denote the $L^p$ norm on the unit circle of the function $g$.
    Next result provides estimates of the $L^2$ and $L^4$ norms of linear combinations of iterates of an inner function.
    It will be applied to finite linear combinations.
    For $t, z \in \C$, let $\langle t, z  \rangle = \Real (\overline{t} z)$ be the standard scalar product in the plane. 

    \begin{theorem}
        \label{thm:L2L4NormEstimates}
        Let $f$ be an inner function with $f(0)=0$ which is not a rotation
        and let $\{a_n \}$ be a sequence of complex numbers with $\sum_n |a_n|^2 < \infty$.
        Consider 
        \begin{equation*}
            \xi = \sum_{n=1}^\infty a_n f^n
        \end{equation*}
        and 
        \begin{equation*}
            {\sigma}^2  = \sum_{n=1}^\infty |a_n|^2 + 2 \Real \sum_{k=1}^{\infty} f'(0)^k \sum_{n=1}^\infty \overline{a_n} a_{n+k} . 
        \end{equation*}
        
        \begin{enumerate}[(a)]
            \item
            \label{thm:item:L2NormEstimate}
            We have $\|\xi \|_2^2 ={\sigma}^2  $ and 
            \begin{equation*}
                C^{-1} \sum_{n=1}^\infty |a_n|^2 \leq {\sigma}^2 \leq  C \sum_{n=1}^\infty |a_n|^2 , 
            \end{equation*}
            where $C= (1+|f'(0)|)(1- |f'(0)|)^{-1}$.
            
            \item
            \label{thm:item:ScalarProduct}
            For any $t \in \C$ we have
            \begin{equation*}
                \ic \langle t, \xi  \rangle^2\, dm = \frac{1}{2} |t|^2 {\sigma}^2. 
            \end{equation*}
            
            \item
            \label{thm:item:L4NormEstimate}
            There exists a constant $C=C(f) >0$ independent of the sequence $\{a_n\}$,
            such that $\|\xi\|_4 \leq C \|\xi\|_2$.
        \end{enumerate}
    \end{theorem}
    \begin{proof}[Proof of Theorem~\ref{thm:L2L4NormEstimates}]
        At almost every point of the unit circle we have 
        \begin{equation}
            \label{eq:ModulusSquaredIdentity}
            |\xi |^2 = \sum_{ n=1}^\infty |a_n|^2 + 2 \Real h,  
        \end{equation}
        where 
        \begin{equation}
            \label{eq:XiCrossProducts}
            h = \sum_{n,k \geq 1 , k>n } \overline{a_n} a_{k} f^k \overline{f^n}.
        \end{equation}
        Part \eqref{lemma:item:IteratesCovariance} of Lemma \ref{lemma:DeconjugateInnerFunction} gives
        \begin{equation*}
            \|\xi\|_2^2 = \sum_{ n=1}^\infty |a_n|^2 + 2 \Real \sum_{n,k \geq 1 , k>n } \overline{a_n} a_{k} f'(0)^{k-n},
        \end{equation*}
        which is the identity in \eqref{thm:item:L2NormEstimate}.
        Next we prove the estimate in \eqref{thm:item:L2NormEstimate}.
        Part \eqref{lemma:item:IteratesCovariance} of Lemma \ref{lemma:DeconjugateInnerFunction} gives that 
        \begin{equation*}
            \ic \overline{f^k} f^n\, dm = b_{n,k},  
        \end{equation*}
        where $b_{n,k} = f'(0)^{n-k}$ if $n \geq k$ and $b_{n,k} = {\overline{f'(0)}}^{k-n}$ if $n < k$.
        Hence
        \begin{equation*}
            \bigg\|\sum_n a_n f^n\bigg\|_2^2 = \sum_{n,k} a_n \overline{a_k} b_{n,k}.   
        \end{equation*}
        Consider the Toeplitz matrix $T$ whose entries are $b_{n,k} = b_{n-k, 0}$, $n,k=1,2,\ldots$ and its symbol
        \begin{equation*}
            s(z) = \sum_{n= -\infty}^\infty  b_{n,0} z^n, \quad z \in \circle. 
        \end{equation*}
        It is well known that $T$ diagonalizes and its eigenvalues are contained
        in the interval in the real line whose endpoints are the essential infimum
        and the essential supremum of $s$.
        See \cite{ref:BottcherGrudskyToeplitzMatrices}.
        Since 
        \begin{equation*}
            s(z)= \frac{1- |f'(0)|^2}{|1- \overline{f'(0)} z|^2}, \quad z \in \circle , 
        \end{equation*}
        the eigenvalues of $T$ are between $C^{-1}$ and $C$.
        This finishes the proof of part \eqref{thm:item:L2NormEstimate}. 
        Since $f(0)=0$, the mean value property gives that 
        \begin{equation*}
            \ic {\xi}^2\, dm =0
        \end{equation*}
        and \eqref{thm:item:ScalarProduct} follows.
        We now prove \eqref{thm:item:L4NormEstimate}.
        Let $C(f)$ denote a positive constant only depending on $f$ whose value may change from line to line.
        The identity \eqref{eq:ModulusSquaredIdentity} gives that at almost every point of the unit circle, we have
        \begin{equation*}
            |\xi|^4 = \left(\sum_{n=1}^\infty |a_n|^2\right)^2 + 4 \Real h \sum_{n=1}^\infty |a_n|^2 + 4 (\Real h)^2,  
        \end{equation*}
        where $h$ is defined in \eqref{eq:XiCrossProducts}.  
        Observe that 
        \begin{equation*}
            \ic h\, dm = \sum_{n,k \geq 1 , k>n } \overline{a_n} a_{k} f'(0)^{k-n}   .  
        \end{equation*}
        Hence Lemma \ref{lemma:AuxiliaryBound} gives that
        \begin{equation}
            \label{eq:XiCrossProductsModulus}
            \left|\ic h\, dm\right| \leq \frac{|f'(0)|}{1- |f'(0)|} \sum_{n=1}^\infty |a_n|^2 .  
        \end{equation}
        Next we will prove that there exists a constant $C=C(f) > 0$ such that
        \begin{equation}
            \label{eq:XiCrossProductsModulusSquared}
            \ic |h|^2\, dm \leq C \left( \sum_{n=1}^\infty |a_n|^2 \right)^2 .  
        \end{equation}
        Observe that \eqref{eq:XiCrossProductsModulus} and \eqref{eq:XiCrossProductsModulusSquared}
        give the estimate in \eqref{thm:item:L4NormEstimate}.
        Write 
        \begin{equation*}
            c_n = \overline{a_n } \sum_{k>n} a_k f^k \overline{f^n} . 
        \end{equation*}
        Using the elementary  identity 
        \begin{equation*}
            \Bigg|\sum_{n=1}^\infty c_n\Bigg|^2
            = \sum_{n=1}^\infty |c_n|^2 + 2 \Real  \sum_{n=1}^\infty  \overline{c_n} \sum_{j >n} c_j , 
        \end{equation*}
        we can write
        \begin{equation*}
            \ic |h|^2\, dm = A + 2 \Real B , 
        \end{equation*}
        where
        \begin{equation*}
            A = \sum_{n=1}^\infty |a_n|^2 \ic \Bigg|\sum_{k>n} a_k f^k\Bigg|^2\, dm
        \end{equation*}
        and 
        \begin{equation}
            \label{eq:BTerms}
            B = \sum_{n=1}^\infty a_n \sum_{k>n } \overline{a_k}  \sum_{j>n} \overline{a_j}  \sum_{l>j} a_l
            \ic \overline{f^k} f^n \overline{f^j} f^l\, dm. 
        \end{equation}
        By part \eqref{thm:item:L2NormEstimate} we have
        \begin{equation*}
            \ic \Bigg|\sum_{k >n } a_k f^k\Bigg|^2\, dm \leq C(f) \sum_{k >n } |a_k|^2
        \end{equation*}
        and we deduce that $A \leq C(f) \left(\sum_{n} |a_n|^2\right)^2$.
        We now estimate $B$.
        If $n<k$ and $n<j<l$, we have
        \begin{equation*}
            \left|\ic  f^n \overline{f^k} \overline{f^j} f^l\, dm\right| = |f'(0)|^{r-n + |l-s|}, 
        \end{equation*}
        where $r= \min \{k,j\}$ and $s= \max \{k,j \}$.
        This estimate follows from last statement in part \eqref{lemma:item:FourArbitraryFactors}
        of Lemma \ref{lemma:FourFactors}.
        Part \eqref{lemma:item:SecondFactorSquared} of Lemma \ref{lemma:FourFactors} gives that 
        \begin{equation*}
            \left|\ic  f^n (\overline{f^k})^2 f^l\, dm\right| \leq C(f) |f'(0)|^ { l-n}, 
        \end{equation*}
        if $n<k<l$.
        The sum over $j> n$ in \eqref{eq:BTerms} will be splited in three terms corresponding to $j > k$, $j=k$ and $j<k$.
        Then $|B|\leq C(f)(B_1 + B_2 + B_3)$ where
        \begin{equation*}
            B_1 = \sum_{n \geq 1} |a_n| \sum_{k>n } |a_k| \sum_{j>k } |a_j|  \sum_{l>j} |a_l| |f'(0)|^{k-n+l-j} , 
        \end{equation*}
        \begin{equation*}
            B_2 =  \sum_{n \geq 1} |a_n|   \sum_{k>n} |a_k|^2   \sum_{l>k } |a_l| |f'(0)|^{l-n}, 
        \end{equation*}
        \begin{equation*}
            B_3 =  \sum_{n \geq 1} |a_n |  \sum_{k>n } |a_k|  \sum_{n<j<k} |a_j|  \sum_{l>j } |a_l|  |f'(0)|^{j-n + |l-k|}. 
        \end{equation*}
        Observe that
        \begin{equation*}
            B_1 =  \sum_{n \geq 1} |a_n|   \sum_{k>n} |a_k|  |f'(0)|^{k-n} \sum_{j>k } |a_j|  \sum_{l>j } |a_l| |f'(0)|^{l-j} . 
        \end{equation*}
        Applying Lemma \ref{lemma:AuxiliaryBound} we deduce that $B_1 \leq C(f) \left( \sum_{n \geq 1} |a_n|^2 \right)^2$.
        Similarly
        \begin{equation*}
            B_2 \leq  \left( \sum_{k \geq 1 } |a_k|^2 \right)   \sum_{n \geq 1 } |a_n|   \sum_{l>n } |a_l| |f'(0)|^{l-n}, 
        \end{equation*}
        which again by Lemma \ref{lemma:AuxiliaryBound} is bounded by $C(f) \left( \sum_{k \geq 1} |a_k|^2 \right)^2$.
        Finally 
        \begin{equation*}
            B_3 =  \sum_{n \geq 1} |a_n |  \sum_{k>n } |a_k|  \sum_{n < j<k} |a_j| |f'(0)|^{j-n}
            \left(\sum_{l >k } |a_l|  |f'(0)|^{ l-k} + \sum_{j<l \leq k} |a_l|  |f'(0)|^{ k-l} \right). 
        \end{equation*}
        Using the trivial estimate
        \begin{equation*}
            \sum_{n<j<k} |a_j| |f'(0)|^{j-n} \leq  \sum_{j>n } |a_j| |f'(0)|^{j-n}, 
        \end{equation*}
        we deduce that $B_3 \leq B_4 + B_5$ where
        \begin{equation*}
            B_4 =  \sum_{n \geq 1} |a_n | \sum_{j>n } |a_j| |f'(0)|^{j-n} \sum_{k>n } |a_k|  \sum_{l>k } |a_l|  |f'(0)|^{ l-k}  
        \end{equation*}
        and
        \begin{equation*}
            B_5 =  \sum_{n \geq 1} |a_n | \sum_{j>n } |a_j| |f'(0)|^{j-n} \sum_{k >n } |a_k|  \sum_{n<l \leq k} |a_l|  |f'(0)|^{ k-l}. 
        \end{equation*}
        Applying Lemma \ref{lemma:AuxiliaryBound} we have $B_4 \leq C(f) \left( \sum_{n \geq 1} |a_n|^2 \right)^2$.
        Writing $t=k-l$ we have
        \begin{equation*}
            \sum_{k > n} |a_k|  \sum_{ l<k} |a_l|  |f'(0)|^{ k-l}
            \leq \sum_{t \geq 1} |f'(0)|^t \sum_{l \geq 1 } |a_l||a_{l+t}|
            \leq \frac{|f'(0)|}{1- |f'(0)|} \sum_{n \geq 1 } |a_n|^2 . 
        \end{equation*}
        We deduce that $B_5 \leq C(f)  \left(\sum_{n \geq 1} |a_n|^2\right)^2$.
        This finishes the proof.
    \end{proof}

    \section{Higher order correlations}
    \label{sec:HigherOrderCorrel}
    Next we will use Aleksandrov-Clark measures to estimate certain integrals
    which will appear in the proof of Theorem \ref{thm:MainTheorem}.
    The main result of this Section is Theorem \ref{thm:HigherCorrelations}.
    We start giving bounds for the size of the iterates $f^n$ and of their derivatives at the origin.

    \begin{lemma}
        \label{lemma:IteratesSizeBound}
        Let $f$ be an analytic mapping from the unit disc into itself with $f(0) = 0$ and $0 < |f'(0)| < 1.$
        Then, there exists an integer $d = d(f) > 0$ such that
        \begin{equation*}
            |f^n(w)| < |f'(0)|^n (1-|w|)^{-d}, \quad w \in \D,
        \end{equation*}
        for every $n \geq 1.$
    \end{lemma}
    \begin{proof}[Proof of Lemma~\ref{lemma:IteratesSizeBound}]
        This is a minor modification of \cite[Lemma~2]{ref:PommerenkeErgodicInnerFunctions}.
        We include the argument for completeness.
        Let us denote $a = |f'(0)|$ and consider the function
        \begin{equation*}
            \psi(w) = w \frac{a+w}{1+aw}, \quad w \in \D,
        \end{equation*}
        denote its $n$-th iterate by $\psi^n$ and observe that, by Schwarz's Lemma and induction, we have
        \begin{equation}
            \label{eq:PommerenkeIterateSizeBound}
            |f^n(w)| \leq \psi^n(|w|), \quad w \in \D,
        \end{equation}
        for every $n \geq 1.$
        Next we use the construction of the Königs function of $\psi$
        (see \cite[pp.~89--93]{ref:ShapiroCompositionOperators}).
        Define for each $n \geq 1$ the function
        \begin{equation*}
            g_n(w) = \frac{1}{a^n} \psi^n(w), \quad w \in \D.
        \end{equation*}
        It is known that $\{g_n\}$ converges uniformly on compact subsets of $\D$
        to $g(w) = w + \ldots$ for $w \in \D,$
        satisfying $g(\psi(w)) = ag(w)$.
        Moreover, for $0 \leq x < 1$ we have that
        \begin{equation*}
            g_{n+1}(x) = \frac{\psi(\psi^n (x))}{a^{n+1}} = g_n(x) \frac{1+a^{n-1}g_n(x)}{1+a^{n+1}g_n(x)} \geq g_n(x),
        \end{equation*}
        so that $g_n(x) \leq g(x)$ for every $n \geq 1.$
    
        Next, since $a > 0,$ there exists $\delta = \delta(f) > 0$ such that $\psi$ is univalent on $\{|w| < \delta\}$
        and, thus, $\psi^n$ and $g_n$ are also univalent in this region by Schwarz's Lemma.
        By Koebe Distortion Theorem, there exists $\varepsilon = \varepsilon(f) > 0$ such that $|g(w)| < 1$ if $|w| < \varepsilon.$
        Now take $x_0 = \varepsilon$ and, for $n \geq 1,$ let $x_{n+1} = \psi^{-1}(x_n).$
        Observe that Schwarz's Lemma implies that $x_{n+1} > x_n$ for every $n \geq 0.$
        Let $d$ be a positive integer that will be determined later on.
        We want to show that
        \begin{equation}
            \label{eq:PommerenkeLimitFunctionBound}
            g(x) < (1-x)^{-d}, \quad 0 \leq x \leq x_n,
        \end{equation}
        for every $n \geq 0.$
        By the choice of $x_0,$ it is clear that \eqref{eq:PommerenkeLimitFunctionBound} holds for $n = 0.$
        Assume that \eqref{eq:PommerenkeLimitFunctionBound} holds for $n$ and let $x_0 \leq x \leq x_{n+1}.$
        By construction, we have that $0 < \psi(x) \leq x_n.$
        Therefore, we get
        \begin{equation*}
            g(x) = \frac{1}{a} g(\psi(x)) < \frac{1}{a} (1-\psi(x))^{-d} =
            \frac{1}{a} \left(\frac{1+ax}{1+x}\right)^d (1-x)^{-d}.
        \end{equation*}
        Since $x \geq x_0 = \varepsilon,$ we get the bound
        \begin{equation*}
            g(x) < \frac{1}{a} \left(\frac{1+a\varepsilon}{1+\varepsilon}\right)^d (1-x)^{-d}.
        \end{equation*}
        Hence, using that $a = |f'(0)| < 1,$ we can choose $d = d(f)$ large enough and independent of $n$ so that
        \eqref{eq:PommerenkeLimitFunctionBound} holds.
        Note that , since $x_n \to 1,$ one has in fact that \eqref{eq:PommerenkeLimitFunctionBound}
        is valid for $0 \leq x < 1.$
        Taking \eqref{eq:PommerenkeIterateSizeBound} and applying \eqref{eq:PommerenkeLimitFunctionBound}, we get
        \begin{equation*}
            |f^n(w)| \leq a^n g_n(|w|) \leq a^n g(|w|) < a^n (1-|w|)^{-d}
        \end{equation*}
        as we wanted to see.
    \end{proof}

    \begin{lemma}
        \label{lemma:IteratesDerivativesBound}
        Let $f$ be an analytic mapping from the unit disc into itself with $f(0)=0$ and $a = |f'(0)| < 1.$
        Let $k,l,n$ be positive integers with $l \leq n$ and consider $g(w) = (f^n (w))^k$ for $w \in \D.$
        Then there exists $n_0 = n_0 (f) >0$ such that for $n \geq n_0$ we have 
        \begin{equation*}
            \frac{|g^{l)}(0)|}{l!} \leq a^{kn/2}  . 
        \end{equation*}
    \end{lemma}
    \begin{proof}[Proof of Lemma~\ref{lemma:IteratesDerivativesBound}]
        Observe first that if $a = 0,$ the result holds trivially.
        Indeed, if $f$ has a zero at the origin of multiplicity $m \geq 1,$
        then $g$ has a zero of multiplicity $km^n$ at the origin.
        Thus, if $m \geq 2$ and $l \leq n,$ we have that $g^{l)}(0) = 0.$

        Assume now that $a > 0.$
        In this case, Lemma \ref{lemma:IteratesSizeBound} asserts that there is a positive integer $d = d(f)$ for which
        \begin{equation*}
            |f^n(w)| \leq a^n (1-|w|)^{-d}, \quad w \in \D,
        \end{equation*}
        for $n=1,2,\ldots.$
        Hence, Cauchy's estimate gives 
        \begin{equation*}
            \frac{|g^{l)}(0)|}{l!} \leq \frac{\max \{|g(w)|: |w|=r \}}{r^l} \leq \frac{a^{kn}}{r^l (1-r)^{kd}}, \quad 0<r<1.  
        \end{equation*}
        Since $l \leq n$ we obtain 
        \begin{equation}
            \label{eq:IteratesDerivativesGrowth}
            \frac{|g^{l)}(0)|}{l!}  \leq  \frac{a^{kn}}{r^n (1-r)^{kd}}, \quad 0<r<1.  
        \end{equation}
        Fix $r$ such that $a^{1/4} < r <1$.
        Then there exists $n_0 = n_0 (f,r)$ such that 
        \begin{equation*}
            \frac{a^{n/2}}{r^n (1-r)^{d}} \leq \frac{a^{n/4}}{(1-r)^{d}} < 1, 
        \end{equation*}
        if $n \geq n_0.$
        Since $k \geq 1 $ we deduce that 
        \begin{equation*}
            \frac{a^{kn/2}}{r^n (1-r)^{kd}} \leq  \frac{a^{n/2}}{r^n (1-r)^{d}}  < 1. 
        \end{equation*}
        Hence, estimate \eqref{eq:IteratesDerivativesGrowth} gives 
        \begin{equation*}
            \frac{|g^{l)}(0)|}{l!}  \leq  a^{kn/2}. 
        \end{equation*}
    \end{proof}

    Let $f$ be an inner function with $f(0)=0$ and let $\{ \mu_\alpha\colon \alpha \in \circle \}$
    be its Aleksandrov-Clark measures.
    Recall that \eqref{eq:ACMeasureHerglotz}  gives that for any $\alpha \in \circle$,
    there exists a constant $C_\alpha \in \R$ such that
    \begin{equation*}
        \frac{\alpha + f(w)}{\alpha - f(w)} = \ic \frac{z+ w}{z - w}\, d \mu_\alpha (z) + iC_\alpha , \quad w \in \D. 
    \end{equation*}
    Expanding both terms in power series, for any positive integer $l$ we have 
    \begin{equation*}
        \ic {\overline{z}}^l\, d\mu_{\alpha}(z) 
        = \sum_{k=1}^l {\overline{\alpha}}^k  \ic f(z)^k {\overline{z}}^l\, dm(z), \quad \alpha \in \circle . 
    \end{equation*}
    Hence for any integer $l$,
    the $l$-th moment of $\mu_\alpha$ is a trigonometric polynomial in the variable $\alpha$ of degree less or equal than $|l|$.
    We will need to estimate the coefficients of this trigonometric polynomial. 

    \begin{lemma}
        \label{lemma:ACMomentsCoefficients}
        Let $f$ be an inner function with $f(0)=0$ and $a = |f'(0)| < 1.$
        Let $l,n$ be integers with $1 \leq |l| \leq n$ and
        let $\{ \mu_\alpha\colon \alpha \in \circle \}$ be the Aleksandrov-Clark measures of $f^n.$
        Then there exists a constant $n_0 = n_0 (f) > 0$ such that if $n \geq n_0$, the coefficients of the trigonometric polynomial 
        \begin{equation*}
            \ic {\overline{z}}^l\, d \mu_{\alpha} (z) 
        \end{equation*}
        are bounded by $a^{n/2}$ for any $\alpha \in \circle$.
    \end{lemma}
    \begin{proof}[Proof of Lemma~\ref{lemma:ACMomentsCoefficients}]
        We can assume $l>0$.
        Then
        \begin{equation*}
            \ic {\overline{z}}^l\, d \mu_{\alpha} (z)
            = \sum_{k=1}^{l} {\overline{\alpha}}^k \frac{g_k^{l)} (0)}{l!}, \quad \alpha \in \circle , 
        \end{equation*}
        where $g_k (w) = (f^n (w))^k$, $w \in \D$.
        Lemma \ref{lemma:IteratesDerivativesBound} gives that 
        \begin{equation*}
            \frac{|g_k^{l)}(0)|}{l!} \leq  a^{kn/2}  
        \end{equation*}
        if  $n$ is sufficiently large.
        Since $k \geq 1$, the proof is completed. 
    \end{proof}

    We are now ready to prove the main result of this Section.
    As before, if $n$ is a positive integer,
    we will use the notation $f^{-n}$ to denote the function defined by $f^{-n} (z) = \overline{f^n (z)}$,
    for almost every $z \in \circle$. 

    \begin{theorem}
        \label{thm:HigherCorrelations}
        Let $f$ be an inner function with $f(0)=0$ and $a = |f'(0)| < 1.$
        Let $1 \leq  k \leq q$ be positive integers.
        Let $\boldsymbol{\varepsilon} = \{\varepsilon_j\}_{j=1}^k$
        where $\varepsilon_j = 1$ or $\varepsilon_j = -1,$
        and let $\boldsymbol{n} = \{n_j\}_{j=1}^k$
        where $n_1 < n_2 < \ldots < n_k$ are positive integers with $n_{j+1} - n_j > q$ for any $j =1,2,\ldots , k-1.$ 
        Consider 
        \begin{equation*}
            I(\boldsymbol{\varepsilon},\boldsymbol{n}) = \left|\ \int_{\circle} \prod_{j=1}^k f^{\varepsilon_j n_j}\, dm \right| 
        \end{equation*}
        Then there exist constants $C= C(f)>0$,
        $q_0 = q_0 (f)>0$ independent of $\boldsymbol{\varepsilon}$ and of $\boldsymbol{n},$
        such that if $q \geq q_0$ we have    
        \begin{equation*}
            I(\boldsymbol{\varepsilon},\boldsymbol{n}) \leq C^k k! a^{ \Phi (\boldsymbol{\varepsilon},\boldsymbol{n}) }, k=1,2, \ldots,  
        \end{equation*}
        where $\Phi (\boldsymbol{\varepsilon},\boldsymbol{n}) = \sum_{j=1}^{k-1} \delta_j (n_{j+1} -n_{j}),$
        with $\delta_j \in \{0,1/2,1\}$ for any $j = 1, \ldots, k-1,$
        and with $\delta_1 = 1$ and $\delta_{k-1} \geq 1/2.$
        In addition, for $j = 2, \ldots, k-1$ the coefficient $\delta_j = 1$ if and only if $\delta_{j-1} = 0.$
        Furthermore, if $\delta_{j-1} > 0,$ the coefficient $\delta_j$
        depends on $\varepsilon_{j+1}, \ldots, \varepsilon_k$ and $n_j, \ldots, n_k$ for $j = 2, \ldots, k-1.$ 
    \end{theorem}
    \begin{proof}[Proof of Theorem~\ref{thm:HigherCorrelations}]
        We first prove the following estimate
        \begin{claim}
            \label{claim1}
            We have 
            \begin{align*}
                I(\boldsymbol{\varepsilon},\boldsymbol{n})   \leq |f'(0)|^{n_2 - n_1} 
                \max \Bigg\{ I(\{\varepsilon_3, \ldots, \varepsilon_k\},\{n_3 - n_2, \ldots, n_k - n_2\}) , \\
                \left|\int_{\circle} z^{2} \prod_{i=3}^k  f^{\varepsilon_i ( n_i - n_2)} (z)\, dm (z) \right|,
                \left|\int_{\circle} z^{-2} \prod_{i=3}^k  f^{\varepsilon_i ( n_i - n_2)} (z)\, dm (z) \right| \Bigg\}
            \end{align*}
        \end{claim}
        To prove Claim \ref{claim1} we can assume $\varepsilon_1 = 1$.
        By Lemma \ref{lemma:DeconjugateInnerFunction} and Aleksandrov Desintegration Theorem we have
        \begin{equation*}
            I(\boldsymbol{\varepsilon},\boldsymbol{n}) 
            =\left| \ic \ic z {\alpha}^{\varepsilon_2} \prod_{i=3}^k f^{\varepsilon_i (n_i - n_2)} (\alpha)\, d \mu_\alpha (z)\, dm(\alpha) \right|, 
        \end{equation*}
        where $\{\mu_\alpha\colon \alpha \in \circle \}$ are the Aleksandrov-Clark measures of $f^{n_2 - n_1}$.
        By \eqref{eq:FirstACMeasureMoment} we have
        \begin{equation*}
            \ic z\, d \mu_\alpha (z) = {\overline{f'(0)}}^{n_2 - n_1} \alpha, \quad \alpha \in \circle . 
        \end{equation*}
        Hence if $\varepsilon_ 2 = -1$, we obtain 
        \begin{equation*}
            I(\boldsymbol{\varepsilon},\boldsymbol{n})
            = a^{n_2 - n_1} I(\{\varepsilon_3, \ldots, \varepsilon_k\},\{n_3 - n_2, \ldots, n_k - n_2\}) 
        \end{equation*}
        and if $\varepsilon_2 = 1$, we obtain
        \begin{equation*}
            I(\boldsymbol{\varepsilon},\boldsymbol{n})
            = a^{n_2 - n_1} \left|\ic z^{2} \prod_{i=3}^k  f^{\varepsilon_i ( n_i - n_2)} (z)\, dm (z)\right| . 
        \end{equation*}
        This proves Claim \ref{claim1}.
        We now prove 
        \begin{claim}
            \label{claim2}
            For any integers $k,l,j$ with $0 < |l| < j$ and $0 < j < k$, we have 
            \begin{align*}
                &\left|\int_{\circle} z^l \prod_{i=j}^k  f^{\varepsilon_i ( n_i - n_{j-1})} (z)\, dm (z) \right| \leq \\
                & \leq j a^{(n_j - n_{j-1})/2} 
                {\max}_{|n| \leq |l|+1 }  \left\{ \left| \ic z^{n} \prod_{i=j+1}^k  f^{\varepsilon_i ( n_i - n_j)} (z)\, dm (z) \right|\right\}
            \end{align*}
        \end{claim}
        By Aleksandrov Desintegration Theorem \eqref{eq:AleksandrovDesintegrationThm} we have
        \begin{equation*}
            \ic z^l \prod_{i=j}^k  f^{\varepsilon_i ( n_i - n_{j-1})} (z)\, dm (z)
            =  \ic \ic z^{l} {\alpha}^{\varepsilon_j} \prod_{i=j+1}^k  f^{\varepsilon_i ( n_i - n_j)} (\alpha)\, d \mu_\alpha (z)\, dm (\alpha),  
        \end{equation*}
        where $\{\mu_\alpha\colon \alpha \in \circle \}$ are the Aleksandrov-Clark measures of $f^{n_j - n_{j-1}}$.
        Since $l \neq 0$, according to Lemma \ref{lemma:ACMomentsCoefficients}, the moment
        \begin{equation*}
            \ic z^{l}\,  d \mu_\alpha (z) , \alpha \in \circle ,   
        \end{equation*}
        is a polynomial in the variable $\alpha$ of degree at most $|l| < j$
        whose coefficients are bounded by $a^{(n_j - n_{j-1})/2}$.
        This proves Claim \ref{claim2}. 

        The proof of Theorem \ref{thm:HigherCorrelations} proceeds as follows.
        We first estimate $I(\boldsymbol{\varepsilon},\boldsymbol{n})$
        by the modulus of one of the three integrals in the right hand side of Claim \ref{claim1} and the factor $a^{n_2-n_1},$
        that corresponds to choosing $\delta_1 = 1.$
        Note that any of these three integrals involve $k-2$ products of iterates of $f.$
        In addition, the integral yielding the maximum in Claim \ref{claim1}
        depends only on $\varepsilon_3, \ldots, \varepsilon_k$ and on $n_2, \ldots, n_k.$
        Now if the integral giving the maximum is the first one, we apply Claim \ref{claim1} again,
        obtaining the factor $a^{n_4-n_3}$ and this gives $\delta_2 = 0$ and $\delta_3 = 1.$
        Otherwise we apply Claim \ref{claim2}, obtaining a factor $2 a^{(n_3-n_2)/2},$
        which corresponds to choosing $\delta_2 = 1/2.$
        Assume that we have applied this procedure to determine the values of $\delta_1, \ldots, \delta_{j-1}.$
        We continue applying Claim \ref{claim1} or \ref{claim2}
        depending on which integral is yielding the maximum in the previous step,
        which depends on $\varepsilon_{j+1}, \ldots, \varepsilon_k$ and $n_j, \ldots, n_k.$
        Observe that when Claim \ref{claim1} is applied,
        the number of factors of iterates of $f$ is reduced by two units and we obtain the factor $a^{n_{j+2} - n_{j+1}},$
        which corresponds to fixing $\delta_j = 0$ and $\delta_{j+1} = 1.$
        When Claim \ref{claim2} is applied, we obtain the factor $j a^{(n_{j+1} -  n_j)/2},$
        corresponding to taking $\delta_j = 1/2,$ and the number of factors of iterates of $f$ is reduced by one unit.
        We continue applying this process at least $k/2$ times and at most $k-2$ times, until we reach integrals of the form
        \begin{equation*}
            \ic z^l f^{\varepsilon_k (n_k - n_{k-1})} (z)\, dm(z), \quad |l| < k-1 
        \end{equation*}
        or
        \begin{equation*}
            \ic f^{\varepsilon_{k-1} (n_{k-1} - n_{k-2})} f^{\varepsilon_k (n_k - n_{k-2})}\, dm .  
        \end{equation*}
        Let $g= f^{n_k - n_{k-1}}$.
        The modulus of the first integral is  $|g^{l)} (0)|/ l!$.
        Since   $|l| < q < n_k - n_{k-1}$, if $q$ is sufficiently large,
        Lemma \ref{lemma:IteratesDerivativesBound} gives that last expression is bounded by $a^{(n_k - n_{k-1})/2}$.
        The modulus of the second integral is bounded by $a^{n_k - n_{k-1}}.$
        This shows that $\delta_{k-1} \geq 1/2$ and concludes the proof.
    \end{proof}

    In the proof of Theorem \ref{thm:MainTheorem}
    we will split the partial sum into finitely many terms
    such that the sum of the variances of these terms
    is asymptotically equivalent to the variance of the initial partial sum.
    Next auxiliary result provides sufficient conditions for this splitting. 

    \begin{lemma}
        \label{lemma:PartialVariances}
        Let $\{a_n \}$ be a sequence of complex numbers and $\lambda \in \C$ with  $|\lambda| < 1$.
        Consider the sequence
        \begin{equation*}
            {\sigma}^2_N = \sum_{ n=1}^N |a_n|^2 + 2 \Real \sum_{k=1}^N {\lambda}^k \sum_{n=1}^{N-k}  \overline{a_n} a_{n+k},
            \quad  N=1,2 \ldots
        \end{equation*}
        For $N>1$, let ${\A}_j = {\A}_j (N)$, $j=1,\ldots , M= M(N)$,
        be pairwise disjoint sets of consecutive positive integers smaller than $N$.
        Consider  
        \begin{equation*}
            {\sigma}^2 ({\A}_j) =
            \sum_{ n \in {\A}_j} |a_n|^2 + 2 \Real \sum_{k \geq 1} {\lambda}^k \sum_{n \in {\A}_j\colon  n+k \in {\A}_j} \overline{a_n} a_{n+k},
            \quad  j=1,2 \ldots , M. 
        \end{equation*}
        Let $\A = \cup {\A}_j$.
        Assume
        \begin{equation}
            \label{eq:PVFirst}
                \lim_{N \to \infty} \frac{\sum_{ n \in \A} |a_n|^2}{\sum_{ n=1}^N |a_n|^2} = 1
        \end{equation}
        and
        \begin{equation}
            \label{eq:PVSec}
            \lim_{j \to \infty} \frac{\max \{ |a_n|^2\colon n \in {\A}_j \} } {\sum_{n \in {\A}_j } |a_n|^2} = 0.
        \end{equation}
        Then 
        \begin{equation*}
            \lim_{N \to \infty} \frac{\sum_{j=1}^M{\sigma}^2 ({\A}_j)}{{\sigma}^2_N} = 1.
        \end{equation*}
    \end{lemma}
    \begin{proof}[Proof of Lemma~\ref{lemma:PartialVariances}]
        Let $\B$ be the set of positive integers smaller or equal to $N$ which are not in the collection $\A$.
        Then
        \begin{equation*}
            {\sigma}^2_N - \sum_{j=1}^M{\sigma}^2 ({\A}_j) = A + 2 \Real B + 2 \Real C , 
        \end{equation*}
        where
        \begin{equation*}
            A =  \sum_{n \in \B } |a_n|^2 , 
        \end{equation*}
        \begin{equation*}
            B = \sum_{k=1}^N {\lambda}^k \sum_{\B (k)} \overline{a_n} a_{n+k},
        \end{equation*}
        where $ \B (k) = \{n \in \B, n \leq N-k \}$ and 
        \begin{equation*}
            C = \sum_{j=1}^M \sum_{k \geq 1} {\lambda}^k \sum_{\A (j,k) } \overline{a_n} a_{n+k} , 
        \end{equation*}
        where $\A(j,k) = \{ n \in {\A}_j\colon n \leq N-k ,  n + k \notin {\A}_j \}$.
        According to part \eqref{thm:item:L2NormEstimate} of Theorem \ref{thm:L2L4NormEstimates} we have
        \begin{equation}
            \label{eq:VarianceLowerEstimate}
            \sigma_N^2 \geq \frac{1 - |\lambda|}{1 + |\lambda|} \sum_{n=1}^N |a_n|^2. 
        \end{equation} 
        Then,
        \begin{equation*}
            \frac{A}{{\sigma}^2_N} \leq \frac{1+ |\lambda|}{1- |\lambda|} \frac{A}{\sum_{n=1 }^N |a_n|^2} 
        \end{equation*}
        which by assumption \eqref{eq:PVFirst}, tends to $0$ as $N \to \infty$. Similarly
        \begin{equation*}
            \frac{|B|}{{\sigma}^2_N} \leq
            \frac{1+ |\lambda|}{1- |\lambda|} \frac{\sum_{k=1}^N |\lambda|^k \sum_{\B (k) } |a_n| |a_{n+k}| }{\sum_{n=1 }^N |a_n|^2} 
        \end{equation*}
        By Cauchy-Schwarz's inequality
        \begin{equation*}
            \sum_{\B (k)} |a_n| |a_{n+k}| \leq  \left(\sum_{n \in \B}  |a_n|^2 \right)^{1/2}\left(\sum_{n=1}^N |a_n|^2 \right)^{1/2} 
        \end{equation*}
        and we deduce
        \begin{equation*}
            \frac{B}{{\sigma}^2_N} \leq
            |\lambda| \frac{1+ |\lambda|}{(1- |\lambda|)^2}
            \frac{\left(\sum_{n \in \B } |a_n|^2 \right)^{1/2} }{\left(\sum_{n=1 }^N |a_n|^2\right)^{1/2}} 
        \end{equation*}
        which according to \eqref{eq:PVFirst}, tends to $0$ as $N \to \infty$.
        We now estimate $C$.
        For any $k \geq 1$, Cauchy-Schwarz's inequality gives
        \begin{equation*}
            \sum_{j=1}^M  \sum_{\A (j,k)  } |a_n| |a_{n+k}| \leq \sum_{n \in \A\colon n \leq N-k   } |a_n| |a_{n+k}|
            \leq \sum_{n=1}^N |a_n|^2 . 
        \end{equation*}
        Hence, applying \eqref{eq:VarianceLowerEstimate}, for any positive integer $k_0 $ we have
        \begin{equation}
            \label{eq:LargeKEstimate}
            \frac{\sum_{k \geq k_0} |\lambda|^k \sum_{j=1}^M  \sum_{n \in {\A} (j,k) } |a_n| |a_{n+k}|}{{\sigma_N}^2}
            \leq |\lambda|^{k_0} \frac{1+ |\lambda|}{(1 - |\lambda|)^2} 
        \end{equation}
        Fix $\varepsilon >0$ and use assumption \eqref{eq:PVSec} to pick $j_0 = j_0 (\varepsilon) >0$ large enough so that 
        \begin{equation}
            \label{eq:SmallCoefficient}
            \sup \{|a_n|^2\colon n \in {\A}_j \} \leq \varepsilon \sum_{n \in {\A}_j} |a_n|^2
        \end{equation}
        if $j >j_0$.
        Pick also $k_0$ such that $|\lambda|^{k_0} < \varepsilon$.
        Fix $k \leq k_0$ and note that there are at most $k$ indices $n \in {\A}_j$ such that $n+k \notin{{\A}_j}$.
        Hence 
        \begin{equation*}
            \sum_{\A (j,k)  } |a_n| |a_{n+k}| \leq k |a_{n_j}| |a_{n_j+k}| ,
        \end{equation*}
        where $n_j = n_j (k) \in {\A}_j$ is the index in ${\A}_j$ with $n_j + k \leq N$,
        where the product $|a_n| |a_{n+k}|$ is maximum. Hence
        \begin{equation*}
            \sum_{j \geq j_0}^M  \sum_{\A (j,k) } |a_n| |a_{n+k}|
            \leq k \sum_{j \geq j_0}^M   |a_{n_j}| |a_{n_{j+k}}|
            \leq k \left(\sum_{j \geq j_0}^M |a_{n_j}|^2\right)^{1/2} \left(\sum_{j=1}^M |a_{n_j + k}|^2\right)^{1/2} . 
        \end{equation*}
        Note that \eqref{eq:SmallCoefficient} gives that 
        \begin{equation*}
            \sum_{j \geq j_0} |a_{n_j}|^2
            \leq \varepsilon \sum_{j \geq j_0} \sum_{n \in {\A}_j} |a_{n}|^2
            \leq \varepsilon \sum_{n=1}^N |a_{n}|^2 .
        \end{equation*}
        Since there are at most $k$ indices $n \in {\A}_j$ such that $n+k \notin{{\A}_j}$, we also have 
        \begin{equation*}
            \sum_{j=1}^M |a_{n_j + k}|^2 \leq k  \sum_{n=1}^N |a_{n}|^2 .
        \end{equation*}
        Applying \eqref{eq:VarianceLowerEstimate} again, we deduce 
        \begin{equation}
            \label{eq:SmallKEstimate}
            \frac{\sum_{k \leq k_0} |\lambda|^k \sum_{j \geq j_0 }^M  \sum_{ \A (j,k) } |a_n| |a_{n+k}|}{{\sigma_N}^2}
            \leq {\varepsilon}^{1/2} \frac{1+ |\lambda|}{1 - |\lambda|} C_1 , 
        \end{equation}
        where $C_1 = \sum_{k \geq 1}|\lambda|^{k} k^{3/2} $.
        The estimates \eqref{eq:LargeKEstimate} and \eqref{eq:SmallKEstimate} give that
        \begin{equation*}
            \frac{|C|}{\sigma_N^2} \leq
            \frac{\sum_{k \leq k_0} |\lambda|^k \sum_{j < j_0 }  \sum_{ \A (j,k) } |a_n| |a_{n+k}|}{\sigma_N^2} + \frac{1+ |\lambda|}{(1- |\lambda|)^2} 
            \varepsilon + C_1 \frac{1+ |\lambda|}{1- |\lambda|} \varepsilon^{1/2}  . 
        \end{equation*}
        This finishes the proof. 
    \end{proof}
    
    We close this Section with an elementary result which will be used in the proof of Theorem \ref{thm:MainTheorem}. 

    \begin{lemma}
        \label{lemma:IntegralLimit}
        Let $\{f_n \}$, $\{g_n \}$ be two sequences of measurable functions defined at almost every point of the unit circle.
        Assume that there exists a constant $C>0$ such that the following conditions hold
        \begin{enumerate}[(a)]
            \item
            $\sup_n \|f_n \|_2 \leq C $ and 
            \begin{equation*}
                \lim_{n \to \infty} \ic f_n\, dm =1
            \end{equation*}
            
            \item
            $ g_n (z) > -C$ for almost every $z \in \circle $ and $\lim_{n \to \infty} \|g_n \|_2 =0.$ 
        \end{enumerate}
        Then 
        \begin{equation*}
            \lim_{n \to \infty} \ic f_n e^{-g_n}\, dm = 1. 
        \end{equation*}
    \end{lemma}
    \begin{proof}[Proof of Lemma~\ref{lemma:IntegralLimit}]
        Cauchy-Schwarz's inequality gives
        \begin{equation*}
            \left|\ic f_n (e^{-g_n} - 1)\, dm \right| \leq \|f_n \|_2 \|e^{-g_n} - 1\|_2. 
        \end{equation*}
        Note that there exists a constant $M=M(C)>0$ such that $|e^{-x} - 1| \leq M |x|$ if $ x \geq -C$.
        Hence $\|e^{-g_n} - 1\|_2 \leq M \|g_n\|_2$, $n=1,2, \ldots$.
        This finishes the proof.
    \end{proof}

    \section{Proof of Theorem~\ref{thm:MainTheorem}}
    \label{sec:ProofMainThm} 
    \begin{proof}[Proof of Theorem~\ref{thm:MainTheorem}]
        Let 
        \begin{equation*}
            S_N^2 = \sum_{n=1}^N |a_n|^2 , \quad N=1,2,\ldots
        \end{equation*}
        Recall that by part \eqref{thm:item:L2NormEstimate} of Theorem \ref{thm:L2L4NormEstimates},
        we have $C^{-1} \sigma_N^2 \leq S_N^2 \leq C \sigma_N^2$, $N=1,2,\ldots$,
        where $C= (1+|f'(0)|)(1-|f'(0)|)^{-1}$.
        Pick $0 < \varepsilon < \eta$,  $p_N = S_N^{1 + \varepsilon}$ and $q_N = S_N^{1 - \varepsilon}$.
        Let $C(f)$ denote a positive constant only depending on $f$ whose value may change from line to line.
        The proof is organized in several steps. 

        \textit{1. Splitting the Sum.}
        In this first step, for $N$ large,
        we will recursively find indices $0 \leq M_k < N_k < M_{k+1} \leq N$, $1 \leq k \leq Q_N$, such that if
        \begin{equation*}
            \xi_k = \sum_{n=M_k +1}^{N_k} a_n f^n , \qquad \eta_k = \sum_{n=N_k +1}^{M_{k+1}} a_n f^n , 
        \end{equation*}
        we have
        \begin{equation}
            \label{eq:QRatios}
            \lim_{N \to \infty} \frac{Q_N}{q_N} = 1, 
        \end{equation}
        \begin{equation}
            \label{eq:ResidueL2Norm}
            \left\|\sum_{n=1}^N a_n f^n - \sum_{k=1}^{Q_N} (\xi_k + \eta_k )\right\|_2^2 \leq  2 C(f) p_N, 
        \end{equation}
        \begin{equation}
            \label{eq:pNqNBounds}
            p_N \leq  \sum_{n=M_k + 1}^{N_k} |a_n|^2 \leq 2p_N,
            \quad  q_N \leq \sum_{n=N_k + 1}^{M_{k+1}} |a_n|^2 \leq 2q_N,
            \quad k=1,2,\ldots , Q_N, 
        \end{equation}
        \begin{equation}
            \label{eq:EtaSumL2NormVanishing}
            \lim_{N \to \infty} \frac{1}{{\sigma}_N} \left\|\sum_{n=1}^N a_n f^n - \sum_{k=1}^{Q_N} \xi_k \right\|_2 = 0, 
        \end{equation}
        \begin{equation}
            \label{eq:IndicesSeparation}
            M_{k+1} - N_k \geq q_N^{\beta}, \quad N_k - M_k \geq p_N^{\gamma}, \quad  k=1,2,\ldots, Q_N - 1,
        \end{equation}
        where  $\beta = (\eta - \varepsilon)(1- \varepsilon)^{-1}$
        and $\gamma = (\eta +  \varepsilon)(1+  \varepsilon)^{-1}$. 

        Pick $M_1 =0$ and let $N_1$ be the smallest positive integer such that
        \begin{equation*}
            \sum_{n=1}^{N_1} |a_n|^2 \geq p_N . 
        \end{equation*}
        The minimality of $N_1$ gives that
        \begin{equation*}
            \sum_{n=1}^{N_1} |a_n|^2 \leq p_N + |a_{N_1}|^2 . 
        \end{equation*}
        Now let $M_2$ be the smallest positive integer such that
        \begin{equation*}
            \sum_{n=N_1 + 1}^{M_2} |a_n|^2 \geq q_N . 
        \end{equation*}
        As before, the minimality of $M_2$ gives that
        \begin{equation*}
            \sum_{n= N_1 + 1}^{M_2} |a_n|^2 \leq q_N + |a_{M_2}|^2 .
        \end{equation*}
        We repeat this process until we arrive at an index $N_k$ or $M_k$ bigger than $N$.
        Let $Q_N $ be the number of times this process is repeated, that is, $k=1,2,\ldots, Q_N$.
        Then
        \begin{equation}\label{eq:ResidueCoefSum}
            R_N = \sum_{M_{Q_N} + 1}^{N} |a_n|^2 \leq 2 p_N .
        \end{equation}
        Since
        \begin{equation*}
            \sum_{n=1}^N a_n f^n - \sum_{k=1}^{Q_N} (\xi_k + \eta_k ) = \sum_{n=M_{Q_N} +1}^{N} a_n f^n ,
        \end{equation*}
        the estimate \eqref{eq:ResidueL2Norm} follows
        from part \eqref{thm:item:L2NormEstimate} of Theorem \ref{thm:L2L4NormEstimates}.
        By construction we have
        \begin{align}
            \label{eq:pNqNAux}
            p_N & \leq \sum_{n=M_k +1}^{N_k} |a_n|^2 \leq p_N + |a_{N_k}|^2\\
            q_N & \leq \sum_{n=N_k +1}^{M_{k+1}} |a_n|^2 \leq q_N + |a_{M_{k+1}}|^2, 
        \end{align}
        for $k=1,2,\ldots , Q_N$.
        Fix $\delta >0$.
        Observe that the assumption \eqref{eq:CoefficientGrowthHpth} gives
        that  $|a_{N_k}|^2+|a_{M_{k+1}}|^2 < \delta q_N$ if $N$ is sufficiently large.
        Taking $\delta <1$ one deduces that \eqref{eq:pNqNBounds} holds if $N$ is sufficiently large.
        Moreover the estimates \eqref{eq:pNqNAux} give that
        \begin{equation}
            \label{eq:pNqNSums}
            (p_N + q_N) Q_N \leq S_N^2 - R_N \leq (1+ \delta) (p_N + q_N) Q_N,
        \end{equation}
        if $N$ is large enough.
        Since $p_N q_N = S_N^2$ and because of the estimate \eqref{eq:ResidueCoefSum},
        \eqref{eq:QRatios} follows from \eqref{eq:pNqNSums}  tending $\delta$ to $0$.
        Observe that 
        \begin{equation}
            \label{eq:BlocCoefSum}
            \sum_{n=N_k + 1}^{M_{k+1}} |a_n|^2 \geq q_N = S_N^{1- \varepsilon}. 
        \end{equation}
        By \eqref{eq:CoefficientGrowthHpth}, if $N$ is sufficiently large,
        we have that $|a_n|^2 \leq  {S_N}^{1 - \eta}$ for any $n \leq N$.
        We deduce from \eqref{eq:BlocCoefSum}
        that ${S_N}^{1 - \eta} (M_{k+1} - N_k) \geq q_N$ and $M_{k+1} - N_k \geq q_N^{\beta}$.
        A similar argument shows that $N_k - M_K \geq p_N^{\gamma}$.
        This proves \eqref{eq:IndicesSeparation}.
        We are now going to prove \eqref{eq:EtaSumL2NormVanishing}.
        Observe that at almost every point of the unit circle we have
        \begin{equation*}
            \left|\sum_{k=1}^{Q_N} \eta_k\right|^2
            = \sum_{k=1}^{Q_N} |\eta_k|^2 + 2 \Real \sum_{k=1}^{Q_N -1 }\sum_{j>k}^{Q_N} \overline{\eta_k} \eta_j. 
        \end{equation*}
        Since $\|\eta_k\|_2^2 \leq 2 C(f) q_N$, we have
        \begin{equation}
            \label{eq:SumOfEtaBlocks}
            \sum_{k=1}^{Q_N} \ic |\eta_k|^2\, dm \leq 2 C(f) q_N Q_N \leq 3C(f) q_N^2,
        \end{equation}
        if $N$ is sufficiently large.
        On the other hand, if $j>k$ we have
        \begin{equation*}
            \left|\ic \overline{\eta_k} \eta_j\, dm \right| \leq \sum |a_r| |a_t| |f'(0)|^{t-r}, 
        \end{equation*}
        where the sum is taken over all indices $r,t$ with $N_k < r \leq M_{k+1}$ and $N_j < t \leq M_{j+1}$.
        Observe that by \eqref{eq:IndicesSeparation}, we have  $t-r \geq p_N^{\gamma}$.
        Writing $l= t - r$ and applying Cauchy-Schwarz's inequality, we obtain 
        \begin{equation*}
            \left| \ic \overline{\eta_k} \eta_j\, dm \right|
            \leq \sum_{l \geq p_N^{\gamma}}  |f'(0)|^{l} \sum  |a_r| |a_{l+r}|
            \leq C(f) |f'(0)|^{p_N^{\gamma}} S_N^2. 
        \end{equation*}
        Hence 
        \begin{equation}
            \label{eq:EtaBlocksCovariances}
            \sum_{k=1}^{Q_N -1 }\sum_{j>k}^{Q_N} \left| \ic \overline{\eta_k} \eta_j\, dm \right|
            \leq C(f) q_N^2 S_N^2 |f'(0)|^{p_N^{\gamma}}. 
        \end{equation}
        Using \eqref{eq:SumOfEtaBlocks} and \eqref{eq:EtaBlocksCovariances} we obtain that
        \begin{equation*}
            \left\|\sum_{k=1}^{Q_N} \eta_k \right\|_2^2 \leq 4 C(f) q_N^2,
        \end{equation*}
        if $N$ is sufficiently large.
        Since $\sigma_N^2 > C(f) S_N^2 = C(f) p_N q_N$, we deduce that 
        \begin{equation}
            \label{eq:EtaBlocksL2Norm}
            \lim_{N \to \infty} \frac{\left\|\sum_{k=1}^{Q_N} \eta_k \right\|_2^2 }{\sigma_N^2} = 0. 
        \end{equation}
        Now \eqref{eq:ResidueL2Norm} and \eqref{eq:EtaBlocksL2Norm} give \eqref{eq:EtaSumL2NormVanishing}.

        The main idea in the rest of the proof is that $\{\eta_k \}$ are irrelevant
        while due to \eqref{eq:IndicesSeparation},  $\{\xi_k \}$ act as independent random variables. 

        \textit{2. Arranging the Fourier Transform.}
        Applying \eqref{eq:EtaSumL2NormVanishing},
        the proof of Theorem \ref{thm:MainTheorem} reduces to show that 
        \begin{equation*}
            T_N = \frac{1}{\sqrt{2} \sigma_N} \sum_{k=1}^{Q_N} \xi_k , \quad N=1,2,\ldots
        \end{equation*}
        converge in distribution to a standard  complex normal variable.
        By the Levi Continuity Theorem, it is sufficient to show that for any complex number $t$ we have 
        \begin{equation}
            \label{eq:FourierTransform}
            \varphi_N (t) = \ic e^{i \langle t , T_N \rangle }\, dm  \to e^{-|t|^2 / 2} ,
            \quad \text{ as } \quad N \to \infty
        \end{equation}
        Here $\langle t,w \rangle = \Real (\overline{t} w)$ is the standard scalar product in the plane.
        In this second step of the proof we will show that 
        \begin{equation}
            \label{eq:FourierTransformApprox}
            \lim_{N \to \infty}  \varphi_N (t) - \ic \prod_{k=1}^{Q_N} \left(1+ \frac{ i \langle t , \xi_k \rangle}{\sqrt{2} \sigma_N} \right)
            \exp{\left( - \frac{{\langle t , \xi_k \rangle}^2}{4 {\sigma_N}^2} \right)}\,  dm   = 0 
        \end{equation}
        Fixed $\delta >0$,
        consider the sets $E_k = \{z \in \circle\colon |\xi_k (z)| > \delta S_N\}$, $k =1,2,\ldots , Q_N$.
        By part \eqref{thm:item:L4NormEstimate} of Theorem \ref{thm:L2L4NormEstimates}
        we have $\|\xi_k\|_4^4 \leq C(f) p_N^2$.
        Txebixeff inequality and \eqref{eq:QRatios} give 
        \begin{equation*}
            \sum_{k=1}^{Q_N} m(E_k) \leq \frac{C(f)p_N^2 Q_N}{{\delta}^4 S_N^4} \leq \frac{2 C(f)}{{\delta}^4 q_N}
        \end{equation*}
        if $N$ is sufficiently large.
        For $\mu > 1$, consider the set
        \begin{equation*}
            E_0 = \set{z \in \circle\colon \sum_{k=1}^{Q_N} {\langle t , \xi_k (z) \rangle}^2 > \mu S_N^2}.  
        \end{equation*}
        By part \eqref{thm:item:L2NormEstimate} of Theorem \ref{thm:L2L4NormEstimates}
        we have $\|\xi_k\|_2 \leq C(f) p_N$.
        Txebixeff inequality and \eqref{eq:QRatios} give 
        \begin{equation*}
            m(E_0) \leq \frac{C(f) |t|^2 Q_N}{\mu q_N} \leq  \frac{2 C(f)|t|^2}{\mu }, 
        \end{equation*}
        if $N$ is sufficiently large.
        Hence the set 
        \begin{equation*}
            E= \bigcup_{k=0}^{Q_N} E_k
        \end{equation*}
        satisfies
        \begin{equation*}
            m(E) \leq 2C(f) \left(\frac{1}{{\delta}^4 q_N} + \frac{|t|^2}{\mu}\right). 
        \end{equation*}
        Using the elementary identity 
        \begin{equation*}
            \exp{(z)} = (1+z) \exp \left(\frac{z^2}{2} + o(|z^2|)\right), 
        \end{equation*}
        where $o(|z|^2)/ |z|^2 \to 0$ as $z \to 0$, we deduce 
        \begin{equation*}
            \exp{( i \langle t , T_N \rangle )}
            = \left(\prod_{k=1}^{Q_N} \left(1+ \frac{ i \langle t , \xi_k \rangle}{\sqrt{2} \sigma_N} \right)
              \exp{\left(-  \frac{{\langle t , \xi_k \rangle}^2}{4 \sigma_N^2}\right) } \right)
              \exp{ \left( \sum_{k=1}^{Q_N} o \left(\frac{{ \langle t , \xi_k \rangle}^2}{\sigma_N^2} \right)\right)}
        \end{equation*}
        Fix $\varepsilon >0$.
        Taking $\delta>0 $ sufficiently small we have
        \begin{equation*}
            \sum_{k=1}^{Q_N} o \left(\frac{{ \langle t , \xi_k (z) \rangle}^2}{\sigma_N^2} \right)
            \leq C(f) \varepsilon \mu , \quad z \in \circle \setminus E. 
        \end{equation*}
        Hence
        \begin{equation*}
            \begin{split}
                &\left|\int_{\circle \setminus E} \exp{(i \langle t , T_N \rangle ) }\, dm
                 - \int_{\circle \setminus E} \prod_{k=1}^{Q_N} \left(1+ \frac{ i \langle t , \xi_k \rangle}{\sqrt{2} \sigma_N}\right)
                   \exp{\left(- \frac{{\langle t , \xi_k \rangle}^2}{4 \sigma_N^2}\right) }\, dm \right| \leq \\
                & \leq \left(e^{C(f) \varepsilon \mu} - 1\right) 
                  \int_{\circle \setminus E} \prod_{k=1}^{Q_N} \left(1+ \frac{{ \langle t , \xi_k \rangle}^2}{ 2 \sigma_N^2} \right)^{1/2}
                  \exp{\left(- \frac{{\langle t , \xi_k \rangle}^2}{4 \sigma_N^2}\right)}\,  dm
                  \leq e^{C(f) \varepsilon \mu} - 1. 
            \end{split}
        \end{equation*}
        Last inequality follows from the elementary estimate $(1+x)^{1/2} e^{-x/2} \leq 1$ if $x \geq 0$.
        Hence
        \begin{equation*}
            \left|\varphi_N (t)
            - \ic \prod_{k=1}^{Q_N} \left(1+ \frac{ i \langle t , \xi_k \rangle}{\sqrt{2} \sigma_N}\right)
                  \exp{- \left(\frac{{\langle t , \xi_k \rangle}^2}{4 \sigma_N^2}\right)}\, dm  \right|
            \leq 2m(E) + e^{C(f) \varepsilon \mu} - 1, 
        \end{equation*}
        which proves \eqref{eq:FourierTransformApprox}.
        Therefore to prove \eqref{eq:FourierTransform} it is sufficient to show that for any $t \in \C$ one has
        \begin{equation*}
            \lim_{N \to \infty} \ic \prod_{k=1}^{Q_N} \left(1+ \frac{ i \langle t , \xi_k \rangle}{\sqrt{2} \sigma_N}\right)
                                    \exp{\left(- \frac{{\langle t , \xi_k \rangle}^2}{4 \sigma_N^2}\right)}\, dm 
            = \exp{(-|t|^2 / 2)}. 
        \end{equation*}
        This will follow from Lemma \ref{lemma:IntegralLimit} applied to the functions
        \begin{equation*}
            f_N = \prod_{k=1}^{Q_N} \left( 1+ \frac{ i \langle t , \xi_k \rangle}{\sqrt{2} \sigma_N} \right), 
        \end{equation*}
        \begin{equation*}
            g_N = \frac{1}{4 \sigma_N^2} \sum_{k=1}^{Q_N} {  \langle t , \xi_k \rangle}^2 - \frac{|t|^2}{2} . 
        \end{equation*}
        According to Lemma \ref{lemma:IntegralLimit} it is sufficient to show
        \begin{equation}
            \label{eq:FunctionFL2Norm}
            \sup_N \|f_N \|_2 < \infty, 
        \end{equation}
        \begin{equation}
            \label{eq:FunctionGL2NormLimit}
            \lim_{N \to \infty}  \|g_N\|_2 =0, 
        \end{equation}
        \begin{equation}
            \label{eq:FunctionFExpectation}
            \lim_{N \to \infty} \ic f_N dm = 1. 
        \end{equation}

        \textit{3. Estimating }$\|f_N\|_2$.
        Observe that
        \begin{equation*}
            \prod_{k=1}^{Q_N} \left(1 + \frac{{  \langle t , \xi_k \rangle}^2}{2\sigma_N^2} \right)
            = 1 + \sum_{k=1}^{Q_N} \frac{1}{2^k \sigma_N^{2k}}
                                   \sum {\langle t , \xi_{j_1} \rangle}^2 \ldots {\langle t , \xi_{j_k} \rangle}^2, 
        \end{equation*}
        where the last sum is taken over all collections of indices $1\leq j_1 < \ldots < j_k \leq Q_N$.
        Since ${  \langle t , \xi_n \rangle}^2 \leq |t|^2 |\xi_n|^2$,
        Theorem \ref{thm:SquaredModuliPartialSums}
        and part \eqref{thm:item:L2NormEstimate} of Theorem  \ref{thm:L2L4NormEstimates} give that
        \begin{equation*}
            \ic {\langle t , \xi_{j_1} \rangle}^2 \ldots {\langle t , \xi_{j_k} \rangle}^2\, dm
            \leq C(f)^k |t|^{2k} p_N^k . 
        \end{equation*}
        Since the total number of distinct collections of indices $j_1, \ldots , j_k$
        verifying $1\leq j_1 < \ldots < j_k \leq Q_N$ is $\binom{Q_N}{k}$, we deduce
        \begin{equation*}
            \ic \prod_{k=1}^{Q_N} \left(1 + \frac{{  \langle t , \xi_k \rangle}^2}{2{\sigma_N}^2} \right)\, dm
            \leq 1 + \sum_{k=1}^{Q_N} {\binom{Q_N}{k}} \frac{C(f)^k |t|^{2k} p_N^k}{2^k \sigma_N^{2k}}. 
        \end{equation*}
        Since $\sigma_N^2 \geq C(f)^{-1} S_N^2 =  C(f)^{-1} p_N q_N$, we deduce
        \begin{equation*}
            \ic \prod_{k=1}^{Q_N} \left(1 + \frac{{  \langle t , \xi_k \rangle}^2}{ 2 \sigma_N^2} \right)\, dm
            \leq 1 + \sum_{k=1}^{Q_N} {\binom{Q_N}{k}} \frac{C(f)^{2k} |t|^{2k}}{2^k q_N^k}
            = \left(1 + \frac{C(f)^2 |t^2|}{2 q_N} \right)^{Q_N}.
        \end{equation*}
        Hence \eqref{eq:QRatios} gives that
        $\|f_N\|_2^2 \leq \exp{(C(f)^2 |t|^2 / 3)}$ if $N$ is sufficiently large.
        This gives \eqref{eq:FunctionFL2Norm}. 

        \textit{4. Estimating }$\|g_N\|_2$.
        Consider the set of indices ${\A}_k = \{n \in \N\colon M_k < n \leq N_k\}$, $k=1, \ldots , Q_N$.
        Then 
        \begin{equation}
            \label{eq:XiBlockNotation}
            \xi_k = \sum_{n \in {\A}_k} a_n f^n , \quad k=1, \ldots , Q_N .
        \end{equation}
        Let $\A = \bigcup_{k=1}^{Q_N} {\A}_k$.
        Observe that \eqref{eq:EtaSumL2NormVanishing} gives 
        \begin{equation*}
            \lim_{N \to \infty} \frac{\sum_{n \in \A} |a_n|^2 }{S_N^2} = 1. 
        \end{equation*}
        This is assumption \eqref{eq:PVFirst} of Lemma \ref{lemma:PartialVariances}.
        Assumption \eqref{eq:PVSec} follows from \eqref{eq:CoefficientGrowthHpth}.
        Thus, Lemma \ref{lemma:PartialVariances} gives 
        \begin{equation}
            \label{eq:XiL2NormSum}
            \lim_{N \to \infty} \frac{\sum_{k=1}^{Q_N} \|\xi_k \|_2^2 }{\sigma_N^2} = 1. 
        \end{equation}
        Denote $\lambda = t/|t|$.
        We have
        \begin{equation*}
            g_N =
            \frac{|t|^2}{4 \sigma_N^2} \sum_{k=1}^{Q_N} \left( 2|\xi_k|^2 + \overline{\lambda^2} \xi_k^2 + \lambda^2 \overline{\xi_k^2} - 2 {\sigma_N}^2 \right). 
        \end{equation*}
        Applying \eqref{eq:XiL2NormSum}, the proof of \eqref{eq:FunctionGL2NormLimit} reduces to show
        \begin{equation*}
            \lim_{N \to \infty} \left\|\frac{1}{\sigma_N^2} \sum_{k=1}^{Q_N} \psi_k \right\|_2 =0, 
        \end{equation*}
        where $\psi_k = 2(|\xi_k|^2 -\|\xi_k\|_2^2)+ \overline{\lambda^2} \xi_k^2 + \lambda^2 \overline{\xi_k^2}\ $ . 
        Now
        \begin{equation}
            \label{eq:PsiSumL2Norm}
            \left\|\sum_{k=1}^{Q_N} \psi_k \right\|_2^2
            = \sum_{k=1}^{Q_N} \| \psi_k \|_2^2 + 2 \Real \sum_{k=1}^{Q_N - 1} \sum_{j>k}^{Q_N} \ic \overline{\psi_k} \psi_j\, dm.
        \end{equation}
        Since $|\psi_k| \leq 4|\xi_k|^2 + 2 \|\xi_k \|_2^2$,
        parts \eqref{thm:item:L2NormEstimate} and \eqref{thm:item:L4NormEstimate}
        of Theorem  \ref{thm:L2L4NormEstimates} give that $\|\psi_k\|_2^2 \leq C(f) p_N^2$.
        Hence
        \begin{equation*}
            \sum_{k=1}^{Q_N} \| \psi_k \|_2^2 \leq C(f) p_N^2 Q_N
        \end{equation*}
        and we deduce
        \begin{equation*}
            \lim_{N \to \infty} \frac{1}{\sigma_N^4} \sum_{k=1}^{Q_N} \| \psi_k \|_2^2 = 0. 
        \end{equation*}
        The second term in \eqref{eq:PsiSumL2Norm} is splitted as
        \begin{equation*}
            \sum_{k=1}^{Q_N - 1} \sum_{j>k}^{Q_N} \ic \overline{\psi_k} \psi_j\, dm = A + B + C + D, 
        \end{equation*}
        where 
        \begin{equation*}
            A = 4 \sum_{k=1}^{Q_N - 1} \sum_{j>k}^{Q_N}
            \ic \left(|\xi_k|^2 - \|\xi_k\|_2^2\right)\left(|\xi_j|^2 - \|\xi_j\|_2^2\right)\, dm , 
        \end{equation*}
        \begin{equation*}
            B = 2 \sum_{k=1}^{Q_N - 1} \sum_{j>k}^{Q_N}
            \ic\left(|\xi_k|^2 - \|\xi_k\|_2^2\right)\left(\overline{\lambda^2}\xi_j^2 + \lambda^2 \overline{\xi_j^2}\right)\, dm , 
        \end{equation*}
        \begin{equation*}
            C = 2 \sum_{k=1}^{Q_N - 1} \sum_{j>k}^{Q_N}
            \ic\left(\overline{\lambda^2} \xi_k^2 + \lambda^2 \overline{\xi_k^2}\right)\left(|\xi_j|^2 - \|\xi_j\|_2^2\right)\, dm ,
        \end{equation*}
        \begin{equation*}
            D = \sum_{k=1}^{Q_N - 1} \sum_{j>k}^{Q_N}
            \ic\left(\overline{\lambda^2} \xi_k^2 + \lambda^2 \overline{\xi_k^2}\right)
               \left(\overline{\lambda^2} \xi_j^2 + \lambda^2 \overline{\xi_j^2}\right)\, dm.  
        \end{equation*}
        By Theorem \ref{thm:SquaredModuliPartialSums},
        $\|\xi_k \xi_j\|_2 = \|\xi_k\|_2 \|\xi_j\|_2$ if $k \neq j$ and we deduce $A=0$.
        Since the mean of $\xi_j^2$ over the unit circle vanishes and at almost every point in the unit circle one has  
        \begin{equation}
            \label{eq:XiModulusSquaredExpression}
            |\xi_k|^2 = \sum_{n \in {\A}_k} |a_n|^2 + 2 \Real \sum_{n \in {\A}_k}\sum_{j \in {\A}_k, j>n} \overline{a_n} a_j \overline{f^n} f^j, 
        \end{equation}
        the integrals in $B$ can be written as a linear combination of 
        \begin{equation*}
            \ic f^{n_1} \overline{f^{j_1}} \left(\overline{\lambda^2} \xi_j^2 + \lambda^2 \overline{\xi_j^2}\right)\, dm,
        \end{equation*}
        where $n_1, j_1 \in {\A}_k$ and
        hence $\max \{n_1, j_1\}< \min \{n\colon n \in {\A}_j\}$.
        According to part \eqref{lemma:item:FourFactorsCancellation} of Lemma \ref{lemma:FourFactors}, 
        \begin{equation*}
            \ic f^{n_1} \overline{f^{j_1}} \xi_j^2\, dm =0 
        \end{equation*}
        and we deduce $B=0$.
        Since the mean of $\xi_k^2$ over the unit circle vanishes, we have
        \begin{equation*}
            C= 4  \Real \overline{{\lambda}^2} \sum_{k=1}^{Q_N - 1} \sum_{j>k}^{Q_N} \ic \xi_k^2 |\xi_j|^2\, dm . 
        \end{equation*}
        For the same reason, using the formula \eqref{eq:XiModulusSquaredExpression}, we have 
        \begin{equation*}
            \ic \xi_k^2 |\xi_j|^2\, dm = \ic \xi_k^2 \Real h_j\,  dm,  
        \end{equation*}
        where
        \begin{equation*}
            h_j= 2  \sum_{r, l \in {\A}_j, l>r} \overline{a_r} a_l \overline{f^r} f^l . 
        \end{equation*}
        Using formula \eqref{eq:XiBlockNotation} to expand $\xi_k^2$, we obtain 
        \begin{equation*}
            \ic \xi_k^2 |\xi_j|^2\, dm = E +F, 
        \end{equation*}
        where
        \begin{equation*}
            E =
            \sum_{n \in {\A}_k} \sum_{r, l \in {\A}_j, l>r} a_n^2
                \ic (f^n)^2 \left( \overline{a_r} a_l \overline{f^r} f^l + a_r \overline{a_l} f^r  \overline{f^l} \right)\, dm,  
        \end{equation*}
        \begin{equation*}
            F =
            2 \sum_{n,s \in {\A}_k\colon s>n} a_n a_s \sum_{r, l \in {\A}_j, l>r} 
                \ic f^n f^s \left( \overline{a_r} a_l \overline{f^r} f^l + a_r \overline{a_l} f^r  \overline{f^l}\right)\, dm.   
        \end{equation*}
        By part \eqref{lemma:item:FirstFactorSquared} of Lemma \ref{lemma:FourFactors} we have
        \begin{equation*}
            \left|\ic (f^n)^2 \overline{f^r} f^l\, dm\right| +
                \left|\ic (f^n)^2 \overline{f^l} f^r\, dm\right|
            \leq C(f) |f'(0)|^{l-n},  \quad \text{ if } n < r < l.   
        \end{equation*}
        We deduce that 
        \begin{equation*}
            |E| \leq C(f) \sum_{n \in {\A}_k} |a_n|^2 \sum_{r, l \in {\A}_j, l>r}  |a_r| |a_l| |f'(0)|^{l-n} . 
        \end{equation*}
        According to  \eqref{eq:IndicesSeparation},
        we have $r - n \geq q_N^{\beta}$ for any $r \in {\A}_j$ and any $n \in {\A}_k$, $j>k$.
        Now 
        \begin{equation*}
            \sum_{r, l \in {\A}_j, l>r}  |a_r| |a_l| |f'(0)|^{l-n}
            \leq |f'(0)|^{q_N^{\beta}} \sum_{t \geq 1} |f'(0)|^{t} \sum_{r \in {\A}_j\colon r+t \in {\A}_j}  |a_r| |a_{r+t}|. 
        \end{equation*}
        By Cauchy-Schwarz's inequality, last sum is bounded by $\sum_{r \in {\A}_j} |a_r|^2 \leq 2 p_N$.
        Hence
        \begin{equation}\label{eq:EstimateE}
            |E| \leq C(f) |f'(0)|^{q_N^\beta } p_N^2
        \end{equation}
        Similarly, part \eqref{lemma:item:FourArbitraryFactors} of Lemma \ref{lemma:FourFactors} gives that 
        \begin{equation*}
            \left|\ic f^n f^s \overline{f^r} f^l\, dm \right|
                + \left|\ic f^n f^s f^r \overline{f^l}\, dm\right|
            \leq C(f) |f'(0)|^{l-n}, \quad n<s<r<l-2, 
        \end{equation*}
        and
        \begin{equation*}
            \left|\ic f^n f^s \overline{f^r} f^l\, dm\right|
                + \left|\ic f^n f^s f^r \overline{f^l}\, dm\right|
            \leq C(f) |f'(0)|^{r-n}, \quad n<s<r<l, r\geq l-2 . 
        \end{equation*}
        Using the trivial estimate $|a_k| \leq S_N$ for any $k \leq N$, we deduce that 
        \begin{equation*}
            |F| \leq C(f) S_N^4 \sum_{n,s \in {\A}_k\colon s>n} \quad \sum_{r, l \in {\A}_j, l>r} |f'(0)|^{l-n}. 
        \end{equation*}
        As before, $l - n \geq q_N^\beta$ for any $r \in {\A}_j$ and any $n \in {\A}_k$, $j>k$.
        We deduce
        \begin{equation*}
            |F| \leq C(f) S_N^4 |f'(0)|^{q_N^\beta / 2}.
        \end{equation*}
        Now, the exponential decay in \eqref{eq:EstimateE} and \eqref{eq:FunctionFL2Norm} give that 
        \begin{equation}
            \lim_{N \to \infty} \frac{C}{\sigma_N^2} = 0. 
        \end{equation}
        The corresponding estimate for $D$ follows from the estimate
        \begin{equation*}
            \left|\ic {\xi_k}^2 \overline{{\xi_j}^2}\, dm\right|
            \leq C(f) S_N^4 |f'(0)|^{q_N^\beta}, \quad k<j. 
        \end{equation*}
        As before this last estimate follows from \eqref{eq:IndicesSeparation} and from
        \begin{equation*}
            \left|\ic f^n f^s \overline{f^l}\overline{f^t}\, dm\right|
            \leq C(f) |f'(0)|^{t-n}, n < s < l < t-2,
        \end{equation*}
        which follows from part \eqref{lemma:item:FourArbitraryFactors} of Lemma \ref{lemma:FourFactors}.
        This finishes the proof of \eqref{eq:FunctionGL2NormLimit}. 

        \textit{5. Integrating} $f_N$.
        In this last step we will prove \eqref{eq:FunctionFExpectation}.
        Observe that at almost every point in the unit circle we have 
        \begin{equation*}
            f_N = 1+ \sum_{k=1}^{Q_N} \frac{i^k}{ 2^{k/2} \sigma_N^k}
                \sum  \langle t , \xi_{i_1} \rangle \ldots  \langle t , \xi_{i_k} \rangle ,  
        \end{equation*}
        where the second sum is taken over all collections of indices $1 \leq i_1 < \ldots i_k \leq Q_N $.
        Fix $1 \leq i_1 < \ldots i_k \leq Q_N $.
        The integral
        \begin{equation*}
            \ic \langle t , \xi_{i_1} \rangle \ldots  \langle t , \xi_{i_k} \rangle\, dm
            = 2^{-k} \ic \prod_{n=1}^k \left(\overline{t} \xi_{i_n} + t \overline{ \xi_{i_n}}\right)\, dm
        \end{equation*}
        is a multiple of a sum of $2^k$ integrals of the form
        \begin{equation*}
            {\overline{t}}^r t^l \ic  \xi_{i_1}^{\varepsilon_1} \ldots  \xi_{i_k}^{\varepsilon_k}\, dm ,
        \end{equation*}
        where $r+l = k$, $\varepsilon_i =1  $ or $\varepsilon_i = -1  $ for $i=1, \ldots , k$
        and we denote ${\xi_i}^{-1} (z) = \overline{\xi_i (z)} $, $ z \in \circle$.
        Now, each $\xi_i$ is a linear combination of iterates of $f$, 
        \begin{equation*}
            \xi_j = \sum_{n \in \A (j)} a_n f^n . 
        \end{equation*}
        Hence
        \begin{equation*}
            \ic \xi_{i_1}^{\varepsilon_1} \ldots  \xi_{i_k}^{\varepsilon_k}\, dm
            = \sum_{\boldsymbol{n} \in \mathcal{C}} \prod_{j=1}^k a_{n_j}^{\varepsilon_j} \ic f^{n_1 \varepsilon_1} \ldots f^{n_k \varepsilon_k}\, dm, 
        \end{equation*}
        where $\sum_{\boldsymbol{n} \in \mathcal{C}}$ means the sum over
        all possible $k$-tuples $\boldsymbol{n} = \{n_j\}_{j=1}^k$ of indices
        such that $n_j \in \A (i_j)$ for $j = 1, \ldots, k.$
        Since $|a_n| \leq S_N$, $n \leq N$, we have
        \begin{equation*}
            \left|\ic \xi_{i_1}^{\varepsilon_1} \ldots  \xi_{i_k}^{\varepsilon_k}\, dm\right|
            \leq S_N^k \sum_{\boldsymbol{n} \in \mathcal{C}} \left| \ic f^{n_1 \varepsilon_1} \ldots f^{n_k \varepsilon_k}\, dm\right|.  
        \end{equation*}
        Let $\boldsymbol{\varepsilon} = \{\varepsilon_j\}_{j=1}^{k-1}$ be fixed and consider
        $\Phi (\boldsymbol{n}) = \Phi(\boldsymbol{\varepsilon},\boldsymbol{n}) = \sum_{j=1}^{k-1}  \delta_j (n_{j+1} - n_j)$
        where $\delta_j \in \{0,1/2,1\}$ for $j = 1, \ldots, k-1,$
        with $\delta_1 = 1$ and $\delta_{k-1} \geq 1/2,$
        and with $\delta_j = 1$ if and only if $\delta_{j-1} = 0$ for $j = 2, \ldots, k-1,$
        as defined in Theorem \ref{thm:HigherCorrelations}.
        Let $a = |f'(0)|.$
        Theorem \ref{thm:HigherCorrelations} gives 
        \begin{equation*}
            \left|\ic \xi_{i_1}^{\varepsilon_1} \ldots  \xi_{i_k}^{\varepsilon_k}\, dm\right|
            \leq k! S_N^k C(f)^k \sum_{\boldsymbol{n} \in \mathcal{C}} a^{\Phi(\boldsymbol{n})}.
        \end{equation*}
        We split the sum over $\boldsymbol{n} \in \mathcal{C}$ as follows.
        Let $\mathcal{D}$ denote the set of $(k-1)$-tuples $\boldsymbol{\delta} = \{\delta_j\}_{j=1}^{k-1}$
        of coefficients that can appear in $\Phi(\boldsymbol{n}).$
        That is, those tuples with $\delta_j \in \{0,1/2,1\}$ for $j = 1, \ldots, k-1,$
        with $\delta_1 = 1$ and $\delta_{k-1} \geq 1/2,$
        and with $\delta_j = 1$ if and only if $\delta_{j-1} = 0,$ for $j = 2, \ldots, k-1.$
        Observe that there are less than $2^k$ such tuples.
        Given a $k$-tuple $\boldsymbol{n} \in \mathcal{C},$
        let us denote by $\boldsymbol{\delta}(\boldsymbol{n})$ the $(k-1)$-tuple $\boldsymbol{\delta}$
        of coefficients appearing in $\Phi(\boldsymbol{n}).$
        Then we have that
        \begin{equation*}
            \sum_{\boldsymbol{n} \in \mathcal{C}} a^{\Phi(\boldsymbol{n})} =
            \sum_{\boldsymbol{\delta} \in \mathcal{D}} \sum_{\{\boldsymbol{n} \in \mathcal{C}\colon \boldsymbol{\delta}(\boldsymbol{n}) =
            \boldsymbol{\delta}\}} a^{\Phi(\boldsymbol{n})}.
        \end{equation*}
        Given $\boldsymbol{\delta} = \{\delta_j\}_{j=1}^{k-1} \in \mathcal{D},$
        we define $\Phi_{\boldsymbol{\delta}}(\boldsymbol{n}) = \sum_{j=1}^{k-1} \delta_j (n_{j+1} - n_j)$
        for every $\boldsymbol{n} \in \mathcal{C}.$
        We clearly have that
        \begin{equation}
            \label{eq:SumFixingDelta}
            \sum_{\boldsymbol{n} \in \mathcal{C}} a^{\Phi(\boldsymbol{n})} \leq
            \sum_{\boldsymbol{\delta} \in \mathcal{D}} \sum_{\boldsymbol{n} \in \mathcal{C}} a^{\Phi_{\boldsymbol{\delta}}(\boldsymbol{n})}.
        \end{equation}
        Consider now a fixed $\boldsymbol{\delta} = (\delta_1,\ldots,\delta_{k-1}),$
        and recall that $\delta_1 = 1.$
        Let $l(1)$ be the minimum integer such that $\delta_{l(1)+1} = 0$
        (we set $l(1) = k-1$ if $\delta_j \neq 0$ for all $1 \leq j \leq k-1$).
        In particular, observe that if $l(1) > 1,$ we have that $\delta_j = 1/2$ for $2 \leq j \leq l(1)$
        by Theorem \ref{thm:HigherCorrelations}.
        Thus, to find a bound for the right-hand side of \eqref{eq:SumFixingDelta}, we need to estimate sums of the form
        \begin{equation}
            \label{eq:PartialSumOverIndices}
            \sum_{j=1}^l \sum_{n_j \in \A(i_j)} a^{(n_2 - n_1) + (n_l - n_2)/2} =
            \sum_{j=1}^l \sum_{n_j \in \A(i_j)} a^{(n_2 - n_1)/2 + (n_l - n_1)/2}
        \end{equation}
        for some $1 < l \leq k-1.$
        Denote here $\overline{n_1} = \max \A(i_1),$ $\underline{n_2} = \min \A(i_2)$ and $\underline{n_l} = \min \A(i_l),$
        and observe that $\underline{n_2} - \overline{n_1} \geq q_N^\beta$ because of \eqref{eq:IndicesSeparation}.
        Assume $l > 2$.
        Summing over $n_1$ and $n_2$ we get that \eqref{eq:PartialSumOverIndices} is bounded by
        \begin{equation*}
            C a^{q_N^\beta/2}\sum_{j=3}^l \sum_{n_j \in \A(i_j)} a^{(n_l - \overline{n_1})/2}.
        \end{equation*}
        Next, summing over $n_j$ for $j$ up to $l-1$ yields the factor $|\A(i_3)| + \ldots + |\A(i_{l-1})|,$
        while summing over $n_l$ we get the factor $a^{(\underline{n_l} - \overline{n_1})/2}.$
        Here, $|\A(i_j)|$ denotes the number of indices in the set $\A(i_j).$
        Using \eqref{eq:IndicesSeparation}, we have that $|\A(i_j)| \geq p_N^\gamma > q_N^\beta$ for any $j = 1, \ldots, k$
        and, thus, we get that $\underline{n_l} - \overline{n_1} \geq q_N^\beta + |\A(i_2)| + \ldots + |\A(i_{l-1})| > lq_N^\beta.$
        Hence, we find that
        \begin{equation}
            \label{eq:BoundForPartialSumOverIndices}
            \sum_{j=1}^l \sum_{n_j \in \A(i_j)} a^{(n_2 - n_1) + (n_l - n_2)/2} \leq C a^{lq_N^\beta/4}.
        \end{equation}
        Note that if $l =1$ or $l = 2$, then \eqref{eq:BoundForPartialSumOverIndices} is obvious.
        Assume now that we have determined $l(m-1).$
        If $l(m-1) < k-1,$ then let $l(m)$ be the minimum integer such that $l(m-1) < l(m) \leq k-1$
        and such that $\delta_{l(m)+1} = 0.$
        We iterate this process until we set $l(r) = k-1$ for some integer $1 \leq r \leq k.$
        Observe that, by Theorem \ref{thm:HigherCorrelations}, we have that $l(m) \geq l(m-1) + 2.$
        Taking $l(0) = 0,$ the full sum over $\boldsymbol{n} \in \mathcal{C}$ in the right-hand side of \eqref{eq:SumFixingDelta}
        becomes a product of sums of the form \eqref{eq:PartialSumOverIndices} with $j$ ranging from $l(m-1)+1$ to $l(m),$
        for $m = 1, \ldots, r.$
        Thus, applying the bound \eqref{eq:BoundForPartialSumOverIndices} we get that
        \begin{equation*}
            \sum_{\boldsymbol{n} \in \mathcal{C}} a^{\Phi_{\boldsymbol{\delta}}(\boldsymbol{n})} \leq
            \prod_{m=1}^r C a^{(l(m)-l(m-1))q_N^\beta/4} \leq
            C^k a^{kq_N^\beta/4}.
        \end{equation*}
        Now, summing over $\boldsymbol{\delta} \in \mathcal{D}$ and using the fact that there are at most $2^k$ such tuples, we get that
        \begin{equation*}
            \sum_{\boldsymbol{n} \in \mathcal{C}} a^{\Phi(\boldsymbol{n})} \leq C^k a^{k q_N^\beta / 4}.
        \end{equation*}
        Thus
        \begin{equation*}
            \left|\ic \xi_{i_1}^{\varepsilon_1} \ldots  \xi_{i_k}^{\varepsilon_k}\, dm\right|
            \leq k! S_N^k C(f)^k a^{k q_N^\beta / 4}.
        \end{equation*}
        We deduce that 
        \begin{equation*}
            \left|\ic \langle t , \xi_{i_1} \rangle \ldots  \langle t , \xi_{i_k} \rangle\, dm\right|
            \leq k! S_N^k C(f)^k |t|^k a^{k q_N^\beta /4}.
        \end{equation*}
        Since the total number of collections of indices $1 \leq i_1 < \ldots < i_k \leq Q_N$ is $\binom{Q_N}{k}$, we deduce that
        \begin{equation*}
            \left|\ic f_N\, dm - 1\right|
            \leq \sum_{k=1}^{Q_N} {\binom{Q_N}{k}} k! 2^{-k/2} \sigma_N^{-k} (C(f)S_N |t|)^k a^{k q_N^\beta / 4}.
        \end{equation*}
        Last sum is smaller than 
        \begin{equation*}
            \left( 1 + \frac{C(f)|t| S_N Q_N  a^{q_N^\beta / 4}}{\sqrt{2} \sigma_N} \right)^{Q_N} - 1,
        \end{equation*}
        which tends to $0$ as $ N \to \infty$ because
        \begin{equation*}
            \lim_{N \to \infty} \frac{S_N Q_N^2 a^{q_N^\beta / 4}}{\sigma_N} = 0. 
        \end{equation*}
    \end{proof}
    
    \printbibliography

@article {ref:AaronsonErgodicInnerFunctions,
    AUTHOR = {Aaronson, Jon},
     TITLE = {Ergodic theory for inner functions of the upper half plane},
   JOURNAL = {Ann. Inst. H. Poincar\'{e} Sect. B (N.S.)},
  FJOURNAL = {Annales de l'Institut Henri Poincar\'{e}. Section B. Calcul des
              Probabilit\'{e}s et Statistique. Nouvelle S\'{e}rie},
    VOLUME = {14},
      YEAR = {1978},
    NUMBER = {3},
     PAGES = {233--253},
      ISSN = {0020-2347},
   MRCLASS = {28D05 (30D50)},
  MRNUMBER = {508928},
MRREVIEWER = {Michael Lin},
}

@article {ref:AleksandrovMeasurablePartitionsCircle,
    AUTHOR = {Aleksandrov, A. B.},
     TITLE = {Measurable partitions of the circle induced by inner
              functions},
   JOURNAL = {Zap. Nauchn. Sem. Leningrad. Otdel. Mat. Inst. Steklov.
              (LOMI)},
  FJOURNAL = {Zapiski Nauchnykh Seminarov Leningradskogo Otdeleniya
              Matematicheskogo Instituta imeni V. A. Steklova Akademii Nauk
              SSSR (LOMI)},
    VOLUME = {149},
      YEAR = {1986},
    NUMBER = {Issled. Line\u{i}n. Teor. Funktsi\u{i}. XV},
     PAGES = {103--106, 188},
      ISSN = {0373-2703},
   MRCLASS = {30D50 (30H05)},
  MRNUMBER = {849298},
MRREVIEWER = {V. Kabaila},
       DOI = {10.1007/BF01665047},
       URL = {https://doi.org/10.1007/BF01665047},
}

@article {ref:AleksandrovMultiplicityBoundaryValuesInnerFunctions,
    AUTHOR = {Aleksandrov, A. B.},
     TITLE = {Multiplicity of boundary values of inner functions},
   JOURNAL = {Izv. Akad. Nauk Armyan. SSR Ser. Mat.},
  FJOURNAL = {Izvestiya Akademii Nauk Armyansko\u{\i} SSR. Seriya Matematika},
    VOLUME = {22},
      YEAR = {1987},
    NUMBER = {5},
     PAGES = {490--503, 515},
      ISSN = {0002-3043},
   MRCLASS = {30D50 (30D40 30D55 46J15 47B38)},
  MRNUMBER = {931885},
MRREVIEWER = {D. Sarason},
}

@article {ref:AleksandrovInnerFunctionsAndRelatedSpaces,
    AUTHOR = {Aleksandrov, A. B.},
     TITLE = {Inner functions and related spaces of pseudocontinuable
              functions},
   JOURNAL = {Zap. Nauchn. Sem. Leningrad. Otdel. Mat. Inst. Steklov.
              (LOMI)},
  FJOURNAL = {Zapiski Nauchnykh Seminarov Leningradskogo Otdeleniya
              Matematicheskogo Instituta imeni V. A. Steklova Akademii Nauk
              SSSR (LOMI)},
    VOLUME = {170},
      YEAR = {1989},
    NUMBER = {Issled. Line\u{i}n. Oper. Teorii Funktsi\u{i}. 17},
     PAGES = {7--33, 321},
      ISSN = {0373-2703},
   MRCLASS = {30D55 (30H05 46E30 47B38)},
  MRNUMBER = {1039571},
MRREVIEWER = {D. Sarason},
       DOI = {10.1007/BF01099304},
       URL = {https://doi.org/10.1007/BF01099304},
}

@article {ref:BaranskiFagellaJarqueKarpinskaAccessesInfinity,
    AUTHOR = {Bara\'{n}ski, Krzysztof and Fagella, N\'{u}ria and Jarque, Xavier and
              Karpi\'{n}ska, Bogus\l awa},
     TITLE = {Accesses to infinity from {F}atou components},
   JOURNAL = {Trans. Amer. Math. Soc.},
  FJOURNAL = {Transactions of the American Mathematical Society},
    VOLUME = {369},
      YEAR = {2017},
    NUMBER = {3},
     PAGES = {1835--1867},
      ISSN = {0002-9947},
   MRCLASS = {37F10 (30D05 30D30)},
  MRNUMBER = {3581221},
MRREVIEWER = {Tarakanta Nayak},
       DOI = {10.1090/tran/6739},
       URL = {https://doi.org/10.1090/tran/6739},
 SHORTHAND = {BFJK17},
}

@article {ref:BaranskiFagellaJarqueKarpinskaEscapingPoints,
    AUTHOR = {Bara\'{n}ski, Krzysztof and Fagella, N\'{u}ria and Jarque, Xavier and
              Karpi\'{n}ska, Bogus\l awa},
     TITLE = {Escaping points in the boundaries of {B}aker domains},
   JOURNAL = {J. Anal. Math.},
  FJOURNAL = {Journal d'Analyse Math\'{e}matique},
    VOLUME = {137},
      YEAR = {2019},
    NUMBER = {2},
     PAGES = {679--706},
      ISSN = {0021-7670},
   MRCLASS = {37F10 (30D05)},
  MRNUMBER = {3938017},
       DOI = {10.1007/s11854-019-0011-0},
       URL = {https://doi.org/10.1007/s11854-019-0011-0},
 SHORTHAND = {BFJK19},
}

@book {ref:BottcherGrudskyToeplitzMatrices,
    AUTHOR = {B\"{o}ttcher, Albrecht and Grudsky, Sergei M.},
     TITLE = {Toeplitz matrices, asymptotic linear algebra, and functional
              analysis},
 PUBLISHER = {Birkh\"{a}user Verlag, Basel},
      YEAR = {2000},
      ISBN = {3-7643-6290-1},
   MRCLASS = {47B35 (15A57 46L99)},
  MRNUMBER = {1772773},
MRREVIEWER = {Stefano Serra Capizzano},
       DOI = {10.1007/978-3-0348-8395-5},
       URL = {https://doi.org/10.1007/978-3-0348-8395-5},
}

@book {ref:CimaMathesonRossCauchyTransform,
    AUTHOR = {Cima, Joseph A. and Matheson, Alec L. and Ross, William T.},
     TITLE = {The {C}auchy transform},
    SERIES = {Mathematical Surveys and Monographs},
    VOLUME = {125},
 PUBLISHER = {American Mathematical Society, Providence, RI},
      YEAR = {2006},
     PAGES = {x+272},
      ISBN = {0-8218-3871-7},
   MRCLASS = {30-02 (30E10 30E20 46E15 46E20 46E22 47B33 47B38)},
  MRNUMBER = {2215991},
MRREVIEWER = {D. Sarason},
       DOI = {10.1090/surv/125},
       URL = {https://doi.org/10.1090/surv/125},
}

@article {ref:ClarkOneDimensionalPerturbations,
    AUTHOR = {Clark, Douglas N.},
     TITLE = {One dimensional perturbations of restricted shifts},
   JOURNAL = {J. Analyse Math.},
  FJOURNAL = {Journal d'Analyse Math\'{e}matique},
    VOLUME = {25},
      YEAR = {1972},
     PAGES = {169--191},
      ISSN = {0021-7670},
   MRCLASS = {47A55},
  MRNUMBER = {301534},
MRREVIEWER = {S. R. Caradus},
       DOI = {10.1007/BF02790036},
       URL = {https://doi.org/10.1007/BF02790036},
}

@article {ref:CraizerEntropyInnerFunctions,
    AUTHOR = {Craizer, M.},
     TITLE = {Entropy of inner functions},
   JOURNAL = {Israel J. Math.},
  FJOURNAL = {Israel Journal of Mathematics},
    VOLUME = {74},
      YEAR = {1991},
    NUMBER = {2-3},
     PAGES = {129--168},
      ISSN = {0021-2172},
   MRCLASS = {30D50 (28D20)},
  MRNUMBER = {1135232},
MRREVIEWER = {Ch. Pommerenke},
       DOI = {10.1007/BF02775784},
       URL = {https://doi.org/10.1007/BF02775784},
}

@article {ref:DoeringMane,
    AUTHOR = {Doering, Claus I. and Ma\~{n}\'e, Ricardo},
     TITLE = {The dynamics of inner functions},
   JOURNAL = {Ensaios Matemáticos},
    VOLUME = {3},
      YEAR = {1991},
     PAGES = {5--79},
       URL = {http://eudml.org/doc/186559},
}

@article {ref:EvdoridouFagellaJarqueSixsmithInnerFunctionSingularities,
    AUTHOR = {Evdoridou, Vasiliki and Fagella, N\'{u}ria and Jarque, Xavier and
              Sixsmith, David J.},
     TITLE = {Singularities of inner functions associated with hyperbolic
              maps},
   JOURNAL = {J. Math. Anal. Appl.},
  FJOURNAL = {Journal of Mathematical Analysis and Applications},
    VOLUME = {477},
      YEAR = {2019},
    NUMBER = {1},
     PAGES = {536--550},
      ISSN = {0022-247X},
   MRCLASS = {30J05 (30D05 30D15 37F10)},
  MRNUMBER = {3950051},
MRREVIEWER = {Crist\'{o}bal Gonz\'{a}lez},
       DOI = {10.1016/j.jmaa.2019.04.045},
       URL = {https://doi.org/10.1016/j.jmaa.2019.04.045},
 SHORTHAND = {EFJS19},
}

@article {ref:FernandezMelianPestanaInnerFunctionsMixing,
    AUTHOR = {Fern\'{a}ndez, José L. and Meli\'{a}n, María V. and Pestana, Domingo},
     TITLE = {Quantitative mixing results and inner functions},
   JOURNAL = {Math. Ann.},
  FJOURNAL = {Mathematische Annalen},
    VOLUME = {337},
      YEAR = {2007},
    NUMBER = {1},
     PAGES = {233--251},
      ISSN = {0025-5831},
   MRCLASS = {37A25 (30D05 30D50 37A05 37F10)},
  MRNUMBER = {2262783},
MRREVIEWER = {P. Lappan},
       DOI = {10.1007/s00208-006-0036-4},
       URL = {https://doi.org/10.1007/s00208-006-0036-4},
}

@article {ref:FernandezMelianPestanaExpandingMaps,
    AUTHOR = {Fern\'{a}ndez, J. L. and Meli\'{a}n, M. V. and Pestana, D.},
     TITLE = {Expanding maps, shrinking targets and hitting times},
   JOURNAL = {Nonlinearity},
  FJOURNAL = {Nonlinearity},
    VOLUME = {25},
      YEAR = {2012},
    NUMBER = {9},
     PAGES = {2443--2471},
      ISSN = {0951-7715},
   MRCLASS = {37A25 (37A05 37B20 37D20)},
  MRNUMBER = {2967113},
MRREVIEWER = {Nathaniel F. G. Martin},
       DOI = {10.1088/0951-7715/25/9/2443},
       URL = {https://doi.org/10.1088/0951-7715/25/9/2443},
}

@article {ref:FernandezPestanaDistortionInnerFunctions,
    AUTHOR = {Fern\'{a}ndez, Jos\'{e} L. and Pestana, Domingo},
     TITLE = {Distortion of boundary sets under inner functions and
              applications},
   JOURNAL = {Indiana Univ. Math. J.},
  FJOURNAL = {Indiana University Mathematics Journal},
    VOLUME = {41},
      YEAR = {1992},
    NUMBER = {2},
     PAGES = {439--448},
      ISSN = {0022-2518},
   MRCLASS = {30C35 (30C85)},
  MRNUMBER = {1183352},
MRREVIEWER = {D. H. Hamilton},
       DOI = {10.1512/iumj.1992.41.41025},
       URL = {https://doi.org/10.1512/iumj.1992.41.41025},
}

@article {ref:GallardoNieminenACMeasures,
    AUTHOR = {Gallardo-Guti\'{e}rrez, Eva A. and Nieminen, Pekka J.},
     TITLE = {The linear fractional model theorem and {A}leksandrov-{C}lark
              measures},
   JOURNAL = {J. Lond. Math. Soc. (2)},
  FJOURNAL = {Journal of the London Mathematical Society. Second Series},
    VOLUME = {91},
      YEAR = {2015},
    NUMBER = {2},
     PAGES = {596--608},
      ISSN = {0024-6107},
   MRCLASS = {30D05},
  MRNUMBER = {3355117},
MRREVIEWER = {A. Hinkkanen},
       DOI = {10.1112/jlms/jdv002},
       URL = {https://doi.org/10.1112/jlms/jdv002},
}

@book {ref:GarnettBoundedAnalyticFunctions,
    AUTHOR = {Garnett, John B.},
     TITLE = {Bounded analytic functions},
    SERIES = {Graduate Texts in Mathematics},
    VOLUME = {236},
   EDITION = {first},
 PUBLISHER = {Springer, New York},
      YEAR = {2007},
      ISBN = {0-387-33621-4},
   MRCLASS = {30D55 (30-02 30H05 46E15 46J15)},
  MRNUMBER = {2261424},
}

@article {ref:LeviNicolauSoler,
    AUTHOR = {Levi, Matteo and Nicolau, Artur and Soler i Gibert, Od\'i},
     TITLE = {Distortion and Distribution of Sets Under Inner Functions},
   JOURNAL = {J. Geom. Anal.},
  FJOURNAL = {Journal of Geometric Analysis},
    VOLUME = {},
      YEAR = {2019},
    NUMBER = {},
     PAGES = {in press},
      ISSN = {},
   MRCLASS = {},
  MRNUMBER = {},
       DOI = {},
       URL = {https://doi.org/10.1007/s12220-019-00236-w},
 SHORTHAND = {LNS19},
}

@article {ref:NeuwirthErgodicity,
    AUTHOR = {Neuwirth, J. H.},
     TITLE = {Ergodicity of some mappings of the circle and the line},
   JOURNAL = {Israel J. Math.},
  FJOURNAL = {Israel Journal of Mathematics},
    VOLUME = {31},
      YEAR = {1978},
    NUMBER = {3-4},
     PAGES = {359--367},
      ISSN = {0021-2172},
   MRCLASS = {28D20 (10K99)},
  MRNUMBER = {516157},
MRREVIEWER = {R. L. Adler},
       DOI = {10.1007/BF02761501},
       URL = {https://doi.org/10.1007/BF02761501},
}

@incollection {ref:PoltoratskiSarasonACMeasures,
    AUTHOR = {Poltoratski, Alexei and Sarason, Donald},
     TITLE = {Aleksandrov-{C}lark measures},
 BOOKTITLE = {Recent advances in operator-related function theory},
    SERIES = {Contemp. Math.},
    VOLUME = {393},
     PAGES = {1--14},
 PUBLISHER = {Amer. Math. Soc., Providence, RI},
      YEAR = {2006},
   MRCLASS = {30D50 (30D55 30E05 30E20 47A45 47A55 47B15 47B33)},
  MRNUMBER = {2198367},
MRREVIEWER = {Javad Mashreghi},
       DOI = {10.1090/conm/393/07366},
       URL = {https://doi.org/10.1090/conm/393/07366},
}

@article {ref:PommerenkeErgodicInnerFunctions,
    AUTHOR = {Pommerenke, Christian},
     TITLE = {On ergodic properties of inner functions},
   JOURNAL = {Math. Ann.},
  FJOURNAL = {Mathematische Annalen},
    VOLUME = {256},
      YEAR = {1981},
    NUMBER = {1},
     PAGES = {43--50},
      ISSN = {0025-5831},
   MRCLASS = {30D05 (28D20)},
  MRNUMBER = {620121},
MRREVIEWER = {Kenneth Stephenson},
       DOI = {10.1007/BF01450942},
       URL = {https://doi.org/10.1007/BF01450942},
}

@incollection {ref:SaksmanACMeasures,
    AUTHOR = {Saksman, Eero},
     TITLE = {An elementary introduction to {C}lark measures},
 BOOKTITLE = {Topics in complex analysis and operator theory},
     PAGES = {85--136},
 PUBLISHER = {Univ. M\'{a}laga, M\'{a}laga},
      YEAR = {2007},
   MRCLASS = {47A55 (30D35 30D40 30D55 30E20 47B33 47B38)},
  MRNUMBER = {2394657},
MRREVIEWER = {D. Sarason},
}

@article {ref:SalemZygmundLacunarySeries-I,
    AUTHOR = {Salem, R. and Zygmund, A.},
     TITLE = {On lacunary trigonometric series},
   JOURNAL = {Proc. Nat. Acad. Sci. U.S.A.},
  FJOURNAL = {Proceedings of the National Academy of Sciences of the United
              States of America},
    VOLUME = {33},
      YEAR = {1947},
     PAGES = {333--338},
      ISSN = {0027-8424},
   MRCLASS = {42.4X},
  MRNUMBER = {22263},
MRREVIEWER = {R. Fortet},
       DOI = {10.1073/pnas.33.11.333},
       URL = {https://doi.org/10.1073/pnas.33.11.333},
}

@article {ref:SalemZygmundLacunarySeries-II,
    AUTHOR = {Salem, R. and Zygmund, A.},
     TITLE = {On lacunary trigonometric series. {II}},
   JOURNAL = {Proc. Nat. Acad. Sci. U.S.A.},
  FJOURNAL = {Proceedings of the National Academy of Sciences of the United
              States of America},
    VOLUME = {34},
      YEAR = {1948},
     PAGES = {54--62},
      ISSN = {0027-8424},
   MRCLASS = {42.4X},
  MRNUMBER = {23936},
MRREVIEWER = {R. Fortet},
       DOI = {10.1073/pnas.34.2.54},
       URL = {https://doi.org/10.1073/pnas.34.2.54},
}

@book {ref:ShapiroCompositionOperators,
    AUTHOR = {Shapiro, Joel H.},
     TITLE = {Composition operators and classical function theory},
    SERIES = {Universitext: Tracts in Mathematics},
 PUBLISHER = {Springer-Verlag, New York},
      YEAR = {1993},
      ISBN = {0-387-94067-7},
   MRCLASS = {47B38 (30C99 46E20 46J15)},
  MRNUMBER = {1237406},
MRREVIEWER = {Aristomenis Siskakis},
       DOI = {10.1007/978-1-4612-0887-7},
       URL = {https://doi.org/10.1007/978-1-4612-0887-7},
}

@article {ref:WeissLILLacunarySeries,
    AUTHOR = {Weiss, Mary},
     TITLE = {The law of the iterated logarithm for lacunary trigonometric
              series},
   JOURNAL = {Trans. Amer. Math. Soc.},
  FJOURNAL = {Transactions of the American Mathematical Society},
    VOLUME = {91},
      YEAR = {1959},
     PAGES = {444--469},
      ISSN = {0002-9947},
   MRCLASS = {42.00},
  MRNUMBER = {108681},
MRREVIEWER = {J.-P. Kahane},
       DOI = {10.2307/1993258},
       URL = {https://doi.org/10.2307/1993258},
}

\end{document}